\pdfoutput=1
\documentclass[
	,12pt%
	,pagesize%
	,headings=small%
	,paper=letterpaper%
	,parskip=false%
	,abstract=true
	,toc=bibliography
	,DIV=12%
]{scrartcl}

\RedeclareSectionCommand[
  beforeskip=1.5ex plus 0.3ex minus 0.1ex,
  afterskip=0.8ex plus 0.1ex
]{section}
\RedeclareSectionCommand[
  beforeskip=1ex plus 0.2ex minus 0.1ex,
  afterskip=0.5ex plus 0.1ex
]{subsection}

\setkomafont{title}{\normalfont}
\setkomafont{subtitle}{\normalfont}
\setkomafont{author}{\normalfont}
\setkomafont{date}{\normalfont}
\addtokomafont{pageheadfoot}{\scshape\small}
\setkomafont{caption}{\footnotesize}
\setkomafont{captionlabel}{\usekomafont{caption}\bfseries}
\setcapindent{0pt}

\usepackage{authblk}

\usepackage[T1]{fontenc}
\usepackage[margin=1in]{geometry}
\usepackage{amsmath}
\usepackage{amssymb}
\usepackage{amsthm}
\usepackage{mathtools}
\usepackage{xspace}
\usepackage[normalem]{ulem} %
\usepackage[dvipsnames]{xcolor}

\usepackage{newtxtext}
\usepackage[varvw]{newtxmath} %

\usepackage{graphicx}
\usepackage{booktabs}
\usepackage[percent]{overpic} 

\usepackage{tikz-cd}
\usepackage{quiver}

\usepackage{textcomp}

\usepackage{enumitem}
\setlist{nosep}

\def\figdir{Pictures/}
\graphicspath{\figdir}

\usepackage{aliascnt} %
\newcommand{\newautotheorem}[3] %
{
\newaliascnt{#1}{#2}
\newtheorem{#1}[#1]{#3}
\aliascntresetthe{#1}
\expandafter\def\csname #1autorefname\endcsname{%
#3%
}%
}

\newautotheorem{lemma}{theorem}{Lemma}
\newautotheorem{proposition}{theorem}{Proposition}
\newautotheorem{corollary}{theorem}{Corollary}

\theoremstyle{definition}
\newautotheorem{definition}{theorem}{Definition}
\newtheorem*{example}{Example}
\newautotheorem{conjecture}{theorem}{Conjecture}
\newautotheorem{remark}{theorem}{Remark}
\newautotheorem{claim}{theorem}{Claim}

\newcommand{\CA}{C\textunderscore{}A}
\newcommand{\AC}{A\textunderscore{}C}
\newcommand{\AleftA}{A\textunderscore{}left\textunderscore{}A}

\newcommand{\Left}{{\mathrm{left}}}
\newcommand{\Right}{{\mathrm{right}}}
\newcommand{\Reverse}{{\mathrm{reverse}}}

\newcommand{\ReAPR}{\texttt{ReAPR}\xspace}

\newcommand{\reapr}{\ReAPR}
\newcommand{\Reapr}{\ReAPR}

\newcommand{\SnapPy}{\texttt{SnapPy}\xspace}
\newcommand{\snappy}{\SnapPy}
\newcommand{\Snappy}{\SnapPy}

\newcommand{\Regina}{\texttt{Regina}\xspace}

\newcommand{\Knoodle}{\texttt{Knoodle}\xspace}

\newcommand{\R}{\mathbb{R}}

\newcommand{\head}{\operatorname{head}}
\newcommand{\tail}{\operatorname{tail}}

\newcommand{\ceq}{\coloneqq}

\let\mgp=\marginpar \marginparwidth18mm \marginparsep1mm
\def\marginpar#1{\mgp{\raggedright\tiny #1}}

\let\lbl=\label
\def\label#1{\lbl{#1}\ifinner\else\marginpar{\ref{#1} #1}\ignorespaces\fi}

\DeclareMathOperator*{\TV}{TV}

\DeclarePairedDelimiterXPP{\pars}[1]{\mathop{}}{\lparen}{\rparen}{}{#1}

\DeclarePairedDelimiterXPP{\abs}[1]{\mathop{}}{\lvert}{\rvert}{}{#1}
\DeclarePairedDelimiterXPP{\norm}[1]{\mathop{}}{\lVert}{\rVert}{}{#1}
\DeclarePairedDelimiterXPP{\seminorm}[1]{\mathop{}}{\lbrack}{\rbrack}{}{#1}
\DeclarePairedDelimiterXPP{\inner}[1]{\mathop{}}{\langle}{\rangle}{}{#1}
\DeclarePairedDelimiterXPP{\brackets}[1]{\mathop{}}{\lbrack}{\rbrack}{}{#1}
\DeclarePairedDelimiterXPP{\braces}[1]{\mathop{}}{\lbrace}{\rbrace}{}{#1}

\DeclarePairedDelimiterXPP{\intervalcc}[1]{\mathop{}}{\lbrace}{\rbrace}{}{#1}
\DeclarePairedDelimiterXPP{\intervalco}[1]{\mathop{}}{\lbrace}{\rparen}{}{#1}
\DeclarePairedDelimiterXPP{\intervaloc}[1]{\mathop{}}{\lparen}{\rbrace}{}{#1}
\DeclarePairedDelimiterXPP{\intervaloo}[1]{\mathop{}}{\lparen}{\rparen}{}{#1}

\DeclarePairedDelimiterXPP{\myset}[2]{\mathop{}}{\lbrace}{\rbrace}{}{#1\,\delimsize\vert\,\mathopen{}#2}
\let\dot\undefined

\DeclarePairedDelimiterXPP{\dot}[2]{\mathop{}}{\langle}{\rangle}{}{#1,#2}

\DeclarePairedDelimiterXPP{\floor}[1]{\mathop{}}{\lfloor}{\rfloor}{}{#1}
\DeclarePairedDelimiterXPP{\ceil}[1]{\mathop{}}{\lceil}{\rceil}{}{#1}

\DeclareDocumentCommand{\converges}{ o }{
	\mathbin{%
		\IfValueTF{#1}{%
			\mathrel{\vbox{\offinterlineskip\ialign{%
				\hfil##\hfil\cr
				$\scriptscriptstyle#1$\cr
				$-\!\!\!-\!\!\!\rightarrow$\cr
			}}}
		}{%
			-\!\!\!-\!\!\!\rightarrow
		}%
	}%
}

\newcommand{\qtext}[1]{\quad\text{#1}\quad}
\newcommand{\qand}{\qtext{and}}

\usepackage[numbers,sort&compress]{natbib}

\usepackage[
	backref  = false%
	,pagebackref = false%
	,bookmarks  = true%
	,bookmarksdepth=3%
]{hyperref}
\usepackage[all]{hypcap} 
\usepackage[capitalise,nameinlink]{cleveref}
\usepackage{enumitem}
\crefname{equation}{}{}

\newcommand{\subjclass}[2][2020]{%
  \par\noindent\emph{#1 Mathematics Subject Classification.} #2\par}
\newcommand{\keywords}[1]{%
  \par\noindent\emph{Key words and phrases.} #1\par}

\begin{document}
\title{Hard unknots are often easy from a different perspective}
\author[1]{Jason Cantarella}
\affil[1]{Mathematics Department, University of Georgia, Athens, GA, USA}
\author[2]{Henrik Schumacher}
\affil[2]{Institute for Mathematics, RWTH Aachen University, Aachen, Germany}
\author[3]{Clayton Shonkwiler}
\affil[3]{Department of Mathematics, Colorado State University, Fort Collins, CO, USA}

\date{\today}
\maketitle

\begin{abstract}
Recent attempts to train AI models to recognize knots have produced millions of ``hard'' unknot diagrams resistant to simplification by Reidemeister moves, pass moves, or random walks on the Reidemeister graph. Most are easy for the methods based on triangulations of the knot complement found in \Regina and \Snappy, but more difficult for simplifiers based on diagrammatic moves. We present \Reapr (Re-embedding And Pass Rerouting), a semi-diagrammatic simplifier which alternates pass-move reduction on a knot diagram with a geometric re-embedding step. The re-embedding minimizes the \emph{total variation} of a height function on the diagram subject to crossing constraints. We show that for an $n$-crossing diagram, the minimum total variation is $2(n-k)$, where~$k$ is the least number of crossings one must virtualize to make the diagram virtually alternating; this is a combinatorial invariant of the diagram. Reprojecting the resulting embedding from a new viewpoint reveals previously hidden simplifications. \Reapr successfully simplifies every published hard-unknot example we are aware of, as well as several new collections ($\approx 2.6$ million examples in total) in under 30 seconds of total CPU time, making it as effective as the best non-diagrammatic unknot recognition method (\Snappy) and about 60$\times$ faster. $\Reapr$ is just as effective at simplifying diagrams of random knots, where it is regularly used to simplify diagrams with tens of millions of crossings. \end{abstract}

{\footnotesize
\begin{quotation}
\subjclass{(57K10; 57-08, 57K12, 05C21, 90C35)}
\keywords{unknot recognition, hard unknots, knot diagram simplification, pass move, total variation, minimum cost flow, virtual knots.}
\end{quotation}}

\section{Introduction}

Given a knot diagram, an interesting computational problem is to transform it to a diagram of the same knot with as few crossings as possible. Deciding whether such a simplification exists, and finding one if so, are problems that go back to Turing~\cite{turingSolvableUnsolvableProblems2004}. This problem has recently attracted a lot of new interest as attempts are made to develop AI methods to recognize and classify knots~\cite{applebaumUnknottingNumberHard2025,GukovLearning,KauffmanDeepKnot,zhangGeneratingKnottedPolymer2026,wangIntegratingGraphConvolutional2024,VandansIKT,lisitsaSupervisedLearningUntangling2023,lindsayLearnabilityKnotInvariants2025,lahoudsleimanGeometricLearningKnot2024,khanReinforcementLearningAlgorithms2022,castelvecchiDeepMindsAIHelps2021,braghettoMachineLearningUnderstands2023}. It is desirable to have a method for this simplification step which is extremely effective, extremely fast, and scalable to very large diagrams. We present \ReAPR (Re-embedding And Pass Rerouting), which reaches all three goals.

To establish some notation, we think of an $n$-crossing knot diagram $D$ as a directed, 4-valent, embedded planar graph $G(D)$ with over and under information at the crossings. The edges of this graph are~\emph{arcs} of the knot $a_0, \dots, a_{2n-1}$; we assume that they are numbered in the order in which we visit them as we follow the knot. A (maximal)~\emph{overpass} is a consecutive set of arcs $a_1, \dots, a_{k}$ where the head of each $a_i$ with $i < k$ goes over at the corresponding crossing while the tail of $a_1$ and the head of $a_k$ go under. A maximal~\emph{underpass} has the same definition, swapping ``over'' and ``under''. 

The first phase of our algorithm finds overpasses and underpasses and reroutes them to have fewer crossings. Such reroutings are called \emph{pass moves}. The second phase constructs a cubic lattice embedding of the knot from the diagram by first using graph-drawing methods to construct a planar grid embedding of the underlying 4-valent embedded planar graph associated to the diagram and then computing a piecewise-constant height function that respects over/under information. The heights are chosen to minimize the \emph{total variation} $\TV(f)$---the sum of absolute height changes between consecutive arcs---subject to the constraint that overarcs are higher than underarcs at each crossing. The total variation has a nice geometric interpretation: it is the integral (over $z$) of the number of times the curve crosses the horizontal plane at height $z$. This means that minimizing TV produces the ``vertically simplest'' presentation of the knot. This constrained optimization problem can be reformulated as a minimum cost flow problem on a directed graph derived from the diagram and solved in near-linear time in the number of crossings.

Surprisingly, the minimum total variation of a diagram turns out to have an attractive topological interpretation: for an $n$-crossing diagram $D$, the minimum possible value of $\TV$ is $2(n-k)$, where $k$ is the smallest number of crossings one must virtualize to make $D$ virtually alternating (\cref{prop:virtualization}). In particular, this energy is a combinatorial invariant of the diagram. The difference from $2n$ is a measure of how far the diagram is from being alternating; for nonalternating diagrams, the optimal height function separates independent regions of the diagram from each other, often putting them on distinct levels. Reprojecting this lattice embedding from a different viewpoint produces a new diagram, often revealing new pass-move simplifications invisible in the original projection.

It is particularly interesting to try to simplify diagrams of the unknot, as a simplification algorithm which was guaranteed to produce a minimal (i.e., 0) crossing diagram would be a solution to the unknot recognition problem~(see~\cite{lackenbyElementaryKnotTheory2017} for a discussion of the problem). Any such algorithm using Reidemeister moves must sometimes involve adding crossings~\cite{BurtonHardUnknots} and any such algorithm using pass moves must sometimes involve pass moves which do not decrease the number of crossings~\cite{yamadaHowFindKnots2000,TuzunSikoraJones}, but we do not know if there is always a weakly decreasing sequence of pass moves. By contrast, it is known that for grid diagrams there is a simplifying sequence which is weakly monotonic in a certain measure of diagram complexity~\cite{DynnikovAP}.

Many authors have compiled tables or databases of complicated diagrams of the unknot~\cite{Ochiai, applebaumUnknottingNumberHard2025, TuzunSikoraJones}. Some of these are extremely challenging for existing codes to untangle. Such diagrams are colloquially referred to as ``hard'' unknots. We provide a highly performant implementation in \texttt{C++}~\cite{Knoodle} which successfully simplifies all published examples. We also simplify some large collections of new examples. See~\cref{sec:effectiveness} for our performance results and a detailed description of our test set, which is publicly available~\cite{testsets}. While we do not claim that \ReAPR solves the unknot recognition problem, we note that it succeeds on every published hard-unknot example we are aware of; we would welcome examples on which it fails.

\section{Pass Moves and Pass Reduction}
\label{sec:pass moves}

The first step in \Reapr is to find and perform reducing pass moves on a given knot diagram until we reach a diagram where no more reducing pass moves exist. We do this with a greedy algorithm, so there is no guarantee that the pass-reduced diagram where we terminate is the best one for the knot. In this section, we describe our algorithm at a level sufficient for conceptual understanding. Those interested in the implementation may also want to consult~\cref{sec:rerouting details}, which gives more details on the algorithm. 

We provide a \texttt{C++} implementation~\cite{Knoodle}, noting that the performance results in~\cref{sec:effectiveness} required very careful attention to memory layout and cache efficiency, which we also describe in~\cref{sec:rerouting details}.

Suppose we find an over- or underpass $a_1, \dots, a_k$ in a diagram $D$. We may imagine deleting it from the diagram, leaving a modified diagram $\hat{D}$ with two distinguished faces $s$ and $t$ consisting of the union of faces to the left and right of $a_1$ and $a_k$, respectively. In $\hat{D}$, the tail crossing of $a_1$ and the head crossing of $a_k$ degenerate to two $T$-junctions. We could then reroute the pass by reconnecting these $T$-junctions along an arc through $\ell$ adjacent faces of $\hat{D}$, as in~\cref{fig:pass move}. Each face boundary crossed produces a new crossing, so if $\ell < k-1$ the rerouting yields an isotopic diagram with fewer crossings. We call this a~\emph{reducing} pass move, and the graph of diagrams connected by pass moves the~\emph{pass graph}. The first part of our algorithm searches for reducing pass moves 
by finding shortest paths from $s$ to $t$ in the dual graph $\hat{D}^*$ (whose vertices are faces of $\hat{D}$ and whose edges connect adjacent faces), performing them in whatever order we find them.

\begin{figure}[t]
\centering
\hfill
\includegraphics[width=0.3\textwidth]{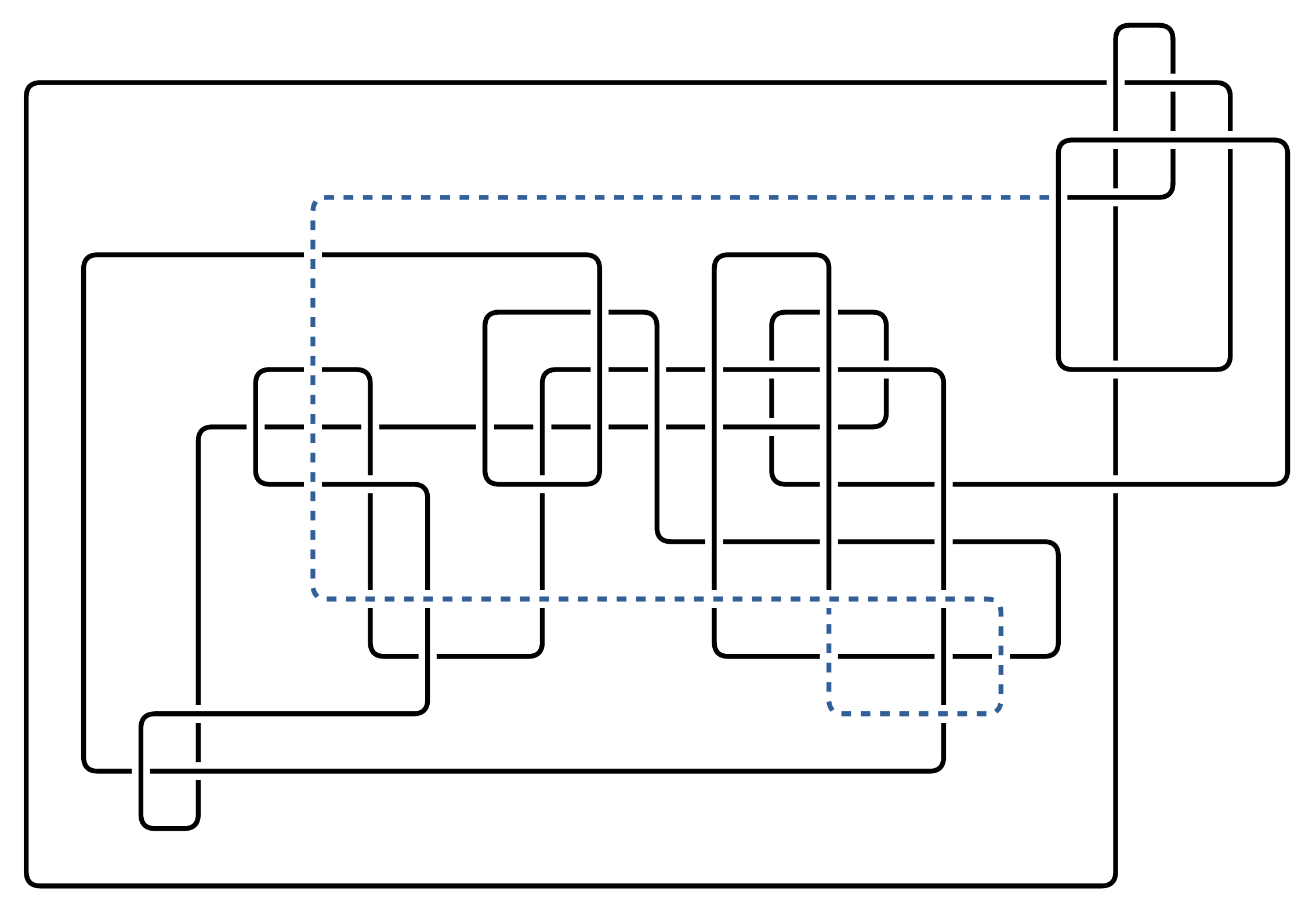}
\hfill
\begin{overpic}[width=0.3\textwidth]{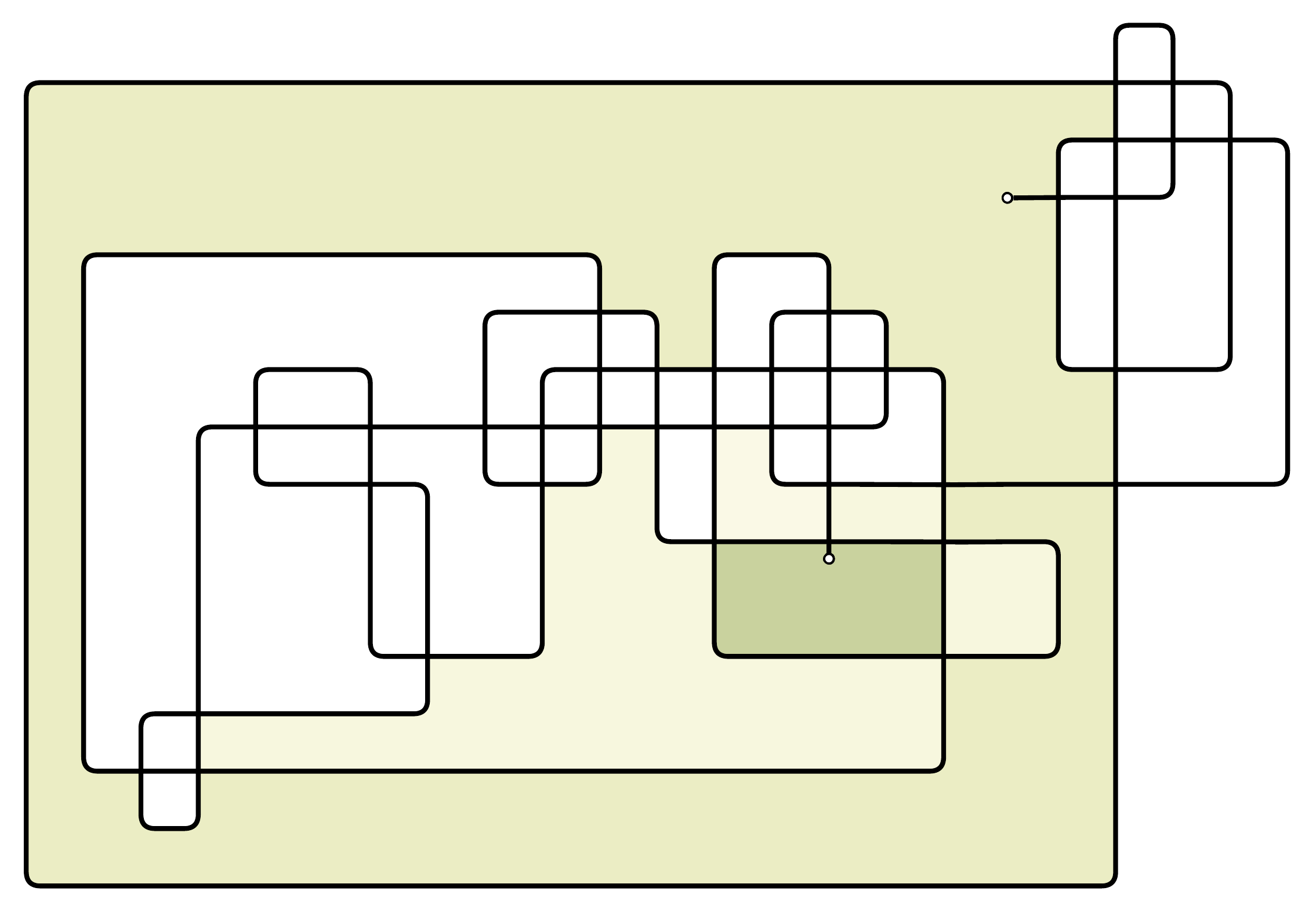}
\put(62,22){{\tiny 1}}
\put(71,53){{\tiny 2}}
\put(50,13){{\tiny 3}}
\put(56,29){{\tiny 4}}
\put(66,29){{\tiny 5}}
\put(74.5,22){{\tiny 6}}
\end{overpic}
\hfill
\includegraphics[width=0.3\textwidth]{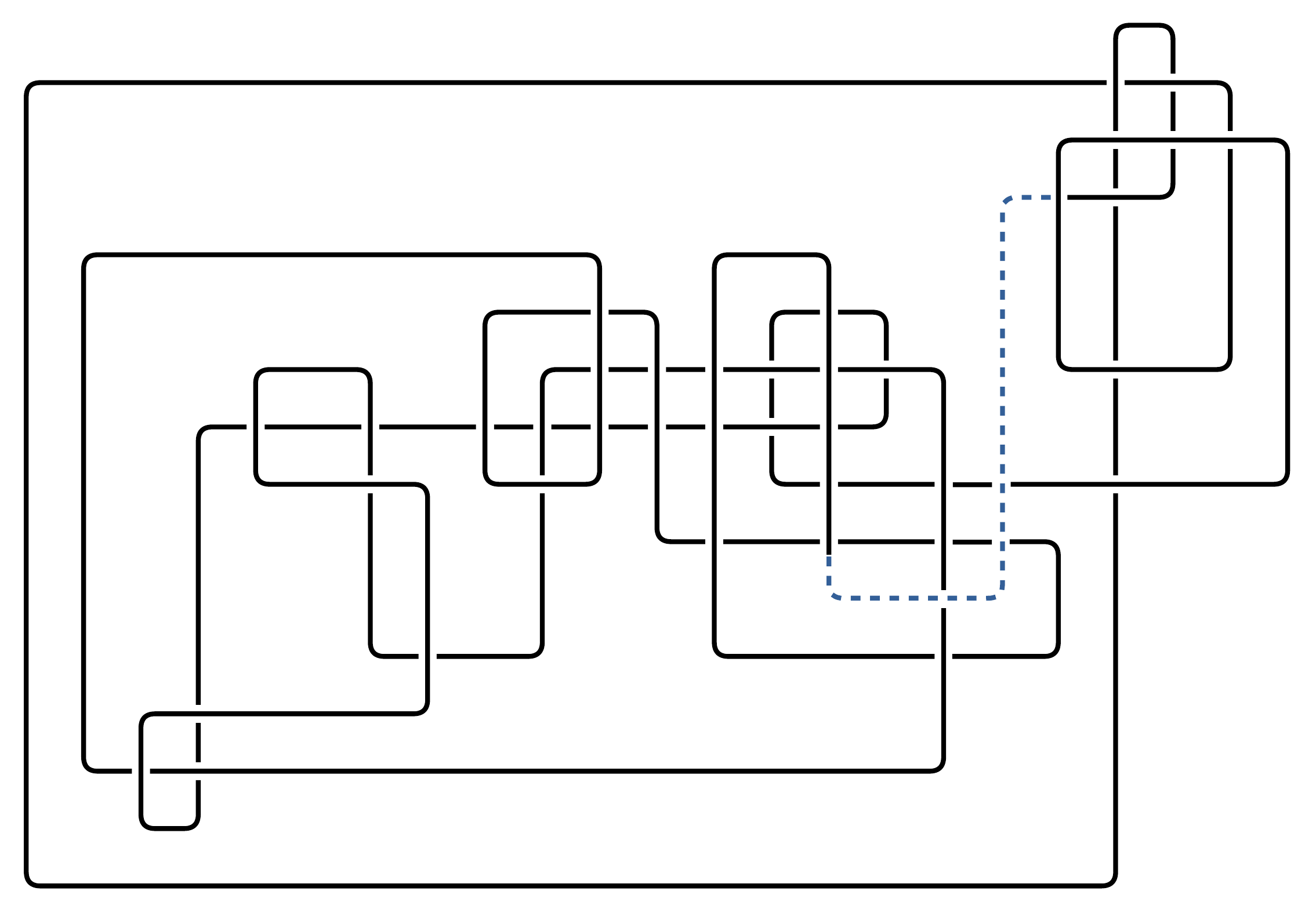}
\hfill
\caption{A simplifying pass move that is easy to discover in the pass graph, but hard to find by Reidemeister moves. The dotted overpass at left has 12 crossings. Bidirectional BFS in the dual graph visits the faces in the labeled order, finding the 3-crossing rerouting at right after visiting 6 faces (on the 7th step, we discover previously visited face 6 while exploring the frontier of face 2). Finding this rerouting by Reidemeister moves would require increasing the crossing count and conducting a large-radius search, as significant portions of the diagram separate the initial and final positions of the arcs.
\label{fig:pass move}}
\end{figure}
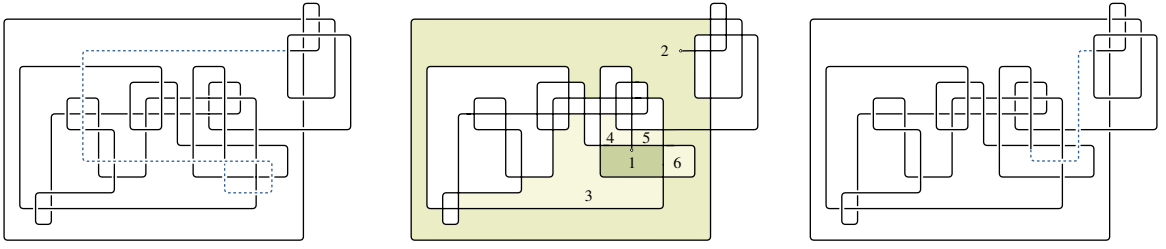

We look for these shortest paths by searching $\hat{D}^*$ using bidirectional breadth-first search (BFS). A natural approach to computing shortest paths in $\hat{D}^*$ (used, for example, by \Snappy{}'s \texttt{Spherogram}) is to explicitly construct $\hat{D}^*$ in $O(n)$ time and search it with a standard graph library. However, the shortest path from $s$ to $t$ has length at most $k-1$, so if $n \gg k$ this constructs far more of the dual graph than we will ever visit. We avoid this by navigating $\hat{D}^*$ implicitly using certain operations on directed arcs of $D$. In what follows, we refer to directed arcs as \emph{darcs}.

Given a darc $da$ in the original diagram $D$, we define operators $\Reverse(\cdot)$ and $\Left(\cdot)$ as follows. First, reversing the direction of $da$ gives $\Reverse(da)$. Next, if $f$ is the face lying to the left of $da$, the \emph{next left} darc $\Left(da)$ is the next darc of $f$ when cycling around the boundary of $f$ in counter-clockwise order. 

Note that faces of $D$ (or vertices of $D^*$) can be identified with the orbits of the operator $\Left(\cdot)$, so we don't need to label them explicitly. Also, given a darc $da$ on the boundary of a face $f$, the darc $\Reverse(da)$ has to its left some face adjacent to~$f$. So by cycling through the darcs of~$f$ and reversing each of them, we can visit neighbor faces of~$f$. Since this visit of neighbors is the only kind of operation required for breadth-first search, we can search for shortest paths in $D^*$ entirely using these two operations on darcs of $D$.

Of course, the goal is to search for shortest paths in $\hat{D}^*$, not in $D^*$. To do so, recall that $\hat{D}$ is formed from $D$ by deleting arcs. So to navigate in $\hat{D}^*$ it suffices to modify the $\Left(\cdot)$ operator by skipping over deleted arcs. This allows us to navigate $\hat{D}^*$ using only operations on darcs of $D$, without ever constructing $\hat{D}$ or $\hat{D}^*$ as separate data structures. 

We now have all the tools needed for bidirectional BFS on $\hat{D}^*$. We perform this search using a \emph{smallest-frontier} heuristic which we now explain. A \emph{frontier} is a set of darcs pointing outward from the set of visited faces: $da$ is in the frontier if the face to the left of $\Reverse(da)$ has been visited but the face to the left of $da$ has not. The initial frontiers $F(s)$ and $F(t)$ consist of the reverses of the darcs in faces $s$ and $t$. To expand a frontier, say $F(s)$: for each $da \in F(s)$, we use $\Left(\cdot)$ to iterate around the face containing $da$, collecting all darcs encountered into a working set $W$. If any darc in $W$ is in $F(t)$, the two searches have met and we reconstruct a shortest path by backtracking. Otherwise, the expanded frontier is a subset of $\{\Reverse(de) : de \in W\}$. 

On a regular $d$-dimensional grid, strictly alternating bidirectional BFS visits a factor of $2^{1-d}$ fewer vertices than unidirectional search. However, the dual graph of a knot diagram is highly irregular: most faces are small, but a few (such as the exterior face of a polygon with many edges) can have tens of thousands of boundary arcs. Strictly alternating between the two frontiers risks expanding an enormous face early on one side, cascading into further large faces. We instead always expand the frontier with fewer darcs, a greedy strategy that avoids this blowup. In our benchmarks, bidirectional BFS with this smallest-frontier heuristic gave roughly a $20\times$ improvement over unidirectional BFS on large diagrams.

\subsection{Pass Moves vs. Reidemeister Moves} 
\label{sub:pass vs R}

Why bother with pass moves at all? After all, both \Snappy and \Regina have algorithms which navigate the Reidemeister graph either by breadth-first search or by random walk until a diagram with fewer crossings is found. Since every pass move is equivalent to a sequence of Reidemeister moves and every Reidemeister move is a pass move, the two graphs have the same connected components; in fact the Reidemeister graph is a subgraph of the pass graph. Therefore, these algorithms are equally powerful; either random walk or breadth-first search will eventually discover any reducing pass move. However, the geometry of the pass graph makes pass reduction much faster. 

Suppose that $D$ and $D'$ are diagrams related by a reducing pass move which takes an overpass of length $k$ to an overpass of length $k'$. It could take $O(n)$ time to discover the overpass, and then we might have to search the dual graph to radius $k' < k$ to discover the reducing edge in the pass graph. The dual graph is planar, so the number of faces examined is usually quite small.\footnote{However, there are link diagrams whose dual graphs contain $\sim k^{\alpha}$ faces in a ball of radius $k$ for any fixed $\alpha > 1$ so it seems hard to bound this step usefully in terms of $k$. For examples, see Section 3.3 of~\cite{benjaminiRecurrenceDistributionalLimits2001}, which constructs plane triangulations with this property. Subdividing each triangle into 3 quadrangles yields planar quadrangulations with the same property; every planar quadrangulation is the dual graph of a link diagram.} 
At worst it is $O(n)$, as there are only $n+2$ faces total. Once the edge is discovered, we can traverse it in $O(k)$ time, transforming $D$ to $D'$. At worst, it will take $O(n^2)$ time to rule out the existence of any reducing pass move. 

In the Reidemeister graph, discovering edges takes $O(n)$ time and traversing them is $O(1)$. This is much faster. However, to find $D'$ by searching the Reidemeister graph, the search~\emph{radius} could be $O(n)$, as we might have to gradually push the overpass over an arbitrarily complex intermediate portion of the diagram as in~\cref{fig:pass move} before we discover the rerouting on the other side. This could require us to generate and examine exponentially many diagrams. This is essentially why pass moves tend to be more efficient. 

\section{Re-embedding}
\label{sec:re-embedding}

Once pass reduction is finished and we have arrived at a diagram $D$ with no reducing pass moves left, we turn to the re-embedding phase. 

\subsection{The Horizontal Component of Re-embedding}
\label{sub:horizontal}

The first step is to construct an orthogonal layout for the embedded planar graph $G(D)$ using a new implementation of Tamassia and Tollis' algorithm~\cite{TamassiaOriginal} 
with the \emph{turn-regularization} and \emph{face saturation} procedures from \cite{zbMATH01455637} 
and with some internal modifications to improve performance and generate more attractive diagrams. 

Finding an embedding of $G(D)$ can be done in linear time \cite{tamassiaPlanarGridEmbedding1989}, but we invest somewhat more time in computing an embedding with a minimal number of bends, which requires solving several linear programming problems during layout. All of these are network flow problems on planar graphs~\cite{TamassiaOriginal}. 
Since our demands and costs are integral and bounded, these network flow problems may be solved in almost linear-time in the number of crossings ($O(n^{1+o(1)})$)~\cite{ChenMaxFlowJACM}. 
In~practice, we use the network simplex algorithm implemented in the class \texttt{MCFSimplex} from the library \texttt{Min-Cost-Flow-Class} project~\cite{MCFClass}, which provides excellent performance and a clean \texttt{C++} interface. 
\texttt{SnapPy}'s \texttt{Spherogram} module and \texttt{Sage} provide Python implementations of comparable graph-drawing methods.

\subsection{The Vertical Component of Re-embedding}
\label{sub:vertical}

Once we have the horizontal layout of $G(D)$, the next goal is to find a height function on over/undercrossings in the knot diagram $D$ which satisfies the constraints imposed by the crossing information and which minimizes the total distance traveled in the $z$-direction while traversing the re-embedded knot. 

Since the height function assigns values to over/undercrossings, let $v_0, \dots, v_{2n-1}$ be the consecutive over/undercrossings as we traverse the knot diagram. If we define a height function $f(v_i)$ so that $a_i$ is at $z$-height $f(v_i)$ when it leaves the tail crossing, then the over/undercrossing information would require $f(v_j) - f(v_i) \geq 1$ whenever $a_j$ passes over $a_i$ at a crossing. Since the arc $a_i$ may need to change height before it arrives at its head crossing, the~\emph{jump} $|f(v_i) - f(v_{i+1})|$ is generally not zero. We can now formally define the total variation, which is the objective function we want $f$ to minimize:

\begin{definition} The \emph{total variation} $\TV(f)$ is given by the total height of the jumps
\[
\TV(f) = |f(v_1) - f(v_0)| + |f(v_2) - f(v_1)| + \cdots + |f(v_0) - f(v_{2n-1})|.
\]
\end{definition}

Therefore, we want to minimize $\TV(f)$ subject to the constraint that $f(v_j) - f(v_i) \geq 1$ whenever $a_j$ crosses over $a_i$. To efficiently find a solution, we will write this as a minimum cost tension (MCT) problem and a corresponding maximum cost flow (MCF) problem. 

To do so, we introduce an auxiliary directed graph $T(D)$ which is closely related to the chord diagram (or Gauss diagram) of $D$. The graph $T(D)$ will be an augmentation of the cycle graph with (ordered) vertex set ${v_0, w_0, v_1, w_1, \dots , v_{2n-1}, w_{2n-1}}$. Conceptually, the added vertices $w_i$ are the slack variables needed to express each absolute value $|f(v_i) - f(v_{i-1})|$ in $\TV(f)$ as the sum of $(f(v_i) - f(v_{i-1}))_+$ and $(f(v_{i-1}) - f(v_i))_+$.

\begin{definition}
\label{def:tension graph} 
The~\emph{tension graph} $T(D)$ of an $n$-crossing knot diagram $D$ is a directed graph with $4n$ vertices and $5n$ edges. Let $v_0, \dotsc, v_{2n-1}$ and $w_0, \dotsc, w_{2n-1}$ be the $4n$ vertices. 
The edges are in the form
\[
\begin{aligned}
v_{i} &\rightarrow w_{i}, \quad i \in \{0, \dotsc, 2n-1\} \\
v_{i+1} &\rightarrow w_{i}, \quad i \in \{0, \dotsc, 2n-2\}, \quad(\text{and } v_{0} \rightarrow w_{2n-1}), \\
v_i &\rightarrow v_j, \quad \text{if arcs $a_i$ and $a_j$ are outgoing from a crossing where  $a_j$ passes over $a_i$}.
\\
\end{aligned}
\]
The edges $v_i \rightarrow v_j$ representing the crossings form a perfect matching between the vertices $v_i$, so we call them~\emph{matching edges} in $T(D)$. We call the other edges~\emph{cycle edges} as, if they were undirected, they would form a $4n$ vertex cycle graph with the $v_i$ and $w_i$.
\end{definition}

We now restate our problem as a constrained optimization problem on the tension graph~$T(D)$:
\begin{equation}
\min_{f} \TV(f) \text{ subject to } f(v_j) - f(v_i) \geq 1 \text{ for each matching edge } v_i \rightarrow v_j.
\label{eq:fundamental}
\end{equation}
We identify functions on the $4n$ vertices of $T(D)$ with vectors in $\R^{4n}$ and functions on the $5n$ edges of $T(D)$ with vectors in $\R^{5n}$. Moreover, we define the cost function $c$ and the lower constraint function $b$ by:
\begin{equation}
	c(e) \ceq 
		\begin{cases}
		0 &\text{if $e$ is a matching edge,}\\
		1 &\text{if $e$ is a cycle edge,}
	\end{cases}
	\qquad 
	b(e) \ceq 
		\begin{cases}
		1 &\text{if $e$ is a matching edge,}\\
		0 &\text{if $e$ is a cycle edge.}
	\end{cases}
	\label{eq:costs and bounds}
\end{equation}
Then we consider the minimum cost tension problem 
\begin{equation}
	\begin{split}
	\min_{f \in \R^{4n}} \sum_{e \in E} c(e) &\pars{f(\head(e)) - f(\tail(e))} \\
	&\text{ subject to }
	f(\head(e)) - f(\tail(e)) \geq b(e)
	\text{ for all $e \in E$.}
	\end{split}	
	\label{eq:mct}
\end{equation}

\begin{lemma}
If $f$ is an optimal solution to~\eqref{eq:mct}, then $f(w_i) = \max(f(v_{i}), f(v_{i+1}))$ and therefore the sum over $e \in \{ v_i \rightarrow w_i, v_{i+1} \rightarrow w_i\}$ of $c(e) \pars{f(\head(e)) - f(\tail(e))}$ is given by $|f(v_i) - f(v_{i+1})|.$
\label{lem:optimax}
\end{lemma}
\begin{proof} 
We know $f(w_i) - f(v_i) \geq b(v_i \rightarrow w_i) = 0$ and $f(w_i) - f(v_{i+1}) \geq b(v_{i+1} \rightarrow w_i) = 0$. Therefore, $f(w_i) \geq f(v_i)$ and $f(w_i) \geq f(v_{i+1})$, so $f(w_i) \geq \max(f(v_{i}), f(v_{i+1}))$. But if $f(w_i) > \max(f(v_{i}), f(v_{i+1}))$ we could reduce the overall cost of $f$ (and maintain feasibility) by lowering $f(w_i)$ to $\max(f(v_{i}), f(v_{i+1}))$. It follows immediately that 
\[
c(v_i \rightarrow w_i) \pars{f(w_i) - f(v_i)} + c(v_{i+1} \rightarrow w_i) \pars{f(w_i) - f(v_{i+1})} = 
|f(v_i) - f(v_{i+1})|.\qedhere
\]
\end{proof}

\begin{lemma}
If $f \in \R^{4n}$ is an optimal solution of~\eqref{eq:mct}, then the cost of $f$ is equal to $\TV(f)$. Further, $ f(\head(e)) - f(\tail(e)) \geq b(e)$ for all edges $e$ of $T(D)$. Therefore, solving~\eqref{eq:mct} is equivalent to solving~\eqref{eq:fundamental}.
\label{lem:mct solves fundamental}
\end{lemma}

\begin{proof}
By~\cref{lem:optimax}, the total cost of the cycle edges is $\sum_i | f(v_i) - f(v_{i+1}) |=  \TV(f)$. Since $c(e) = 0$ for all matching edges, this is the total cost of $f$. Further, for each matching edge, $b(v_i \rightarrow v_j) = 1$ so $f(v_j) - f(v_i) \geq 1$, as required.
Conversely, any feasible solution $f$ of \eqref{eq:fundamental} extends to a feasible solution of \eqref{eq:mct} with cost $\TV(f)$ by setting $f(w_i) \ceq \max(f(v_i), f(v_{i+1}))$. So the two problems have the same minimum value, and restricting an optimal solution of \eqref{eq:mct} to the $v_i$ solves \eqref{eq:fundamental}.
\end{proof}

\begin{example} \label{example:bilevel}
	A valid height function $f$ is given by
	\begin{equation}
		f(v_i) \ceq 	
	\begin{cases} 
		 0& \text{if $a_i$ is outgoing as an underarc,} \\
		 1& \text{if $a_i$ is outgoing as an overarc.}
	\end{cases}
	\label{eq:bilevel}
	\end{equation}
\end{example}
This height function is too flat to make good embeddings for \reapr, as projections of these embeddings to new diagrams tend to simply recover the original diagram. We will see below that if $D$ is alternating, no improvement is possible. However, if $D$ is nonalternating, we may reduce $\TV$ further. To see this, we now transform the MCT problem to its dual MCF problem. Again, we identify functions on the $5n$ edges of $T(D)$ with vectors in $\R^{5n}$.
 
\begin{proposition}
The minimum cost tension problem~\eqref{eq:MinimumCostTensionProblem} is the linear programming dual of the following maximum cost flow problem where $g \in \R^{5n}$ is a vector of (non-negative) flows on edges:
\begin{equation}
\max_{g \in \R^{5n}} \!\!\!\!\!\! \sum_{e \text{\vphantom{$|^|$} a matching edge}} \!\!\!\!\!\!\!\!\!\!\! g(e), 
\text{ subject to }
\begin{cases}
g(e) \geq 0 & \text{ for all $e$,} \\
\sum\limits_{\{e \mid v = \head(e)\}} \!\!\!\!\!\! g(e) - \!\!\!\!\!\! \sum\limits_{\{e \mid v = \tail(e)\}} \!\!\!\!\!\! g(e) + s(v) = 0 &
\text{ for all $v$.} \\
\end{cases}
\label{eq:MaximumCostFlow}
\end{equation}
where $s(v) = +2$ if $v$ is one of the $v_i$ and $s(v) = -2$ if $v$ is one of the $w_i$. That is, each $v_i$ is a supply vertex with supply $2$ and each $w_i$ is a demand vertex with demand $2$.
\end{proposition}

\begin{proof}
Let $A$ be the $5n \times 4n$ incidence matrix\footnote{This matrix represents the discrete exterior derivative on the directed graph $T(D)$.} of the directed graph $T(D)$. That is, $A_{e,\head(e)} = 1$, $A_{e,\tail(e)} = -1$ for each edge in $T(D)$, and all other entries are zero. With the cost and bound vectors $c$ and $b$ from~\eqref{eq:costs and bounds}, the minimum cost tension problem~\eqref{eq:mct} can be written as follows:
\begin{equation}
	\min_{f \in \R^{4n}}
	c^\transp A \, f
	\quad
	\text{subject to}
	\quad
	A \, f \geq b
	.
	\label{eq:MinimumCostTensionProblem}
\end{equation}
The LP dual is
\begin{equation}
	\max_{g \in \R^{5n}}
	b^\transp  g
	\quad
	\text{subject to}
	\quad
	A^\transp g = A^\transp c
	\qand
	g \geq 0.
	\label{eq:MaximumCostFlowProblem}
\end{equation}
We interpret $g$ as a flow on $T(D)$: the constraint $A^\transp g = A^\transp c$ requires net flow $(A^\transp c)_v$ into each vertex $v$. Since each $v_i$ is the tail of exactly two cycle edges and the head of none, $(A^\transp c)_{v_i} = -2$; it follows that $v_i$ is a~\emph{supply}. Likewise the inflow $(A^\transp c)_{w_i} = +2$, so $w_i$ is a~\emph{demand}. The objective $b^\transp g$ counts only flow on matching edges. This is the claimed maximum cost flow problem.
\end{proof}

The formulation in terms of an MCT problem and its dual allows us to draw some conclusions about optimal solutions.

\begin{proposition}
In an $n$-crossing diagram $D$, any height function $f$ solving~\eqref{eq:fundamental} has $\TV(f) \leq 2n$. Moreover, $D$ is alternating if and only if $\TV(f) = 2n$.
\label{prop:mintv and alternating}
\end{proposition}

\begin{proof}
The example height function~\eqref{eq:bilevel} is a feasible solution to the convex programming problem~\eqref{eq:fundamental}; therefore an optimal solution must exist. This solution has $|f(v_i) - f(v_{i-1})| \leq 1$ for all $i$, so $\TV(f) \leq 2n$. In fact, $\TV(f)$ is equal to the number of times $a_i$ switches between being an overarc and underarc as we traverse the diagram, so it is $< 2n$ if $D$ is nonalternating and equal to $2n$ if $D$ is alternating. 

This last shows that an optimal solution $f$ to~\eqref{eq:MinimumCostTensionProblem} for an alternating diagram $D$ has objective function $c^T A f \leq 2n$. We now demonstrate a feasible solution $g$ for~\eqref{eq:MaximumCostFlowProblem} with objective function value $b^T g = 2n$. This tells us that an optimal solution has $b^T g \geq 2n$. By strong duality, the objective functions for the dual problems~\eqref{eq:MinimumCostTensionProblem} and~\eqref{eq:MaximumCostFlow} have the same value at an optimal solution; so this will suffice to show that this optimal value is $2n$. But this is also the total variation of the solution to~\eqref{eq:fundamental} by~\cref{lem:mct solves fundamental}.

Here is our feasible solution. For each matching edge $v_i \rightarrow v_j$, let us assign flows 
\[\begin{tikzcd}[ampersand replacement=\&]
	{w_i} \&\& {w_j} \\
	{v_i} \&\& {v_j} \\
	{w_{i-1}} \&\& {w_{j-1}}
	\arrow["0"', from=2-1, to=1-1]
	\arrow["{+2}", from=2-1, to=2-3]
	\arrow["0", from=2-1, to=3-1]
	\arrow["{+2}", from=2-3, to=1-3]
	\arrow["{+2}"', from=2-3, to=3-3]
\end{tikzcd}\]

Since the $v_i \rightarrow v_j$ are a perfect matching, this assigns flows to all edges. Further, it clearly obeys the conservation constraints at $v_i$ and $v_j$. To see that it obeys the conservation constraints at all other vertices, note that since $D$ is alternating the $v_i$ alternate between being heads and tails of their matching edges as we proceed around the cycle. Therefore, each $w_k$ is adjacent to exactly one $v_l$ which has incoming flow from its matching edge. The corresponding $v_l \rightarrow w_k$ edge supplies flow~$2$ (exactly what we need to balance the demand at $w_k$) while the other incoming edge for $w_k$ carries no flow at all.
\end{proof}

We can now describe the solution to the minimization problem quite explicitly. 

\begin{lemma}
There exists an optimal solution to~\eqref{eq:MaximumCostFlow} where each edge carries flow of $0$ or $2$. \label{lem:zero-two}
\end{lemma}

\begin{proof}
The constraint matrix of \eqref{eq:MaximumCostFlow} is the incidence matrix of the directed graph $T(D)$, which is totally unimodular (this goes back to Poincar\'e, see~\cite{Poincare1900}). By the theorem of Hoffman and Kruskal~\cite{HoffmanKruskal}, a linear program with totally unimodular constraint matrix and integral right-hand side attains its optimum, if it has one, at an integral point. Our right-hand side is not merely integral but even, so consider the flow problem obtained from \eqref{eq:MaximumCostFlow} by halving all supplies and demands to $\pm 1$. Scaling flows by $2$ is a bijection between the feasible sets of the halved and original problems which doubles the objective, so it matches optimal solutions to optimal solutions. The halved problem is feasible (route flow $\tfrac{1}{2}$ along every cycle edge and none along the matching edges) and its objective is bounded, so by Hoffman--Kruskal it has an integral optimal solution $g'$. Then $g \ceq 2g'$ is an optimal solution of \eqref{eq:MaximumCostFlow} in which every flow value is an even integer. Finally, every flow value of $g$ is at most $2$: each cycle edge carries at most $2$ because the flows into its head $w_i$ are non-negative and sum to exactly $2$, and each matching edge carries at most $2$ by conservation at its tail. Hence each flow value is $0$ or~$2$.
\end{proof}

This gives us the following characterization of the minimal total variation energy of a diagram, which is rather knot-theoretic in flavor, but is most easily expressed in terms of virtual knot theory. Classically, each crossing in a knot diagram comes with over and under information. Virtual knot theory adds a third type of crossing, the ``virtual crossing'', which is neither over nor under. The ``virtualization'' of a crossing replaces an ordinary crossing by a virtual crossing. A virtually alternating diagram is one where the ordinary crossings of the knot alternate between over and under, but the knot may also pass through any number of virtual crossings between ordinary crossings. An example alternating virtualization of the nonalternating knot $8_{19}$ appears in~\cref{fig:819}. 

\begin{proposition}
Suppose that $D$ is an $n$-crossing diagram and $k$ is the smallest number for which we can convert $D$ into a virtually alternating diagram by virtualizing $k$ crossings. Any optimal solution to~\eqref{eq:fundamental} has $\TV(f) = 2(n-k)$. There is a corresponding height function with no more than $n-k$ maxima and $n-k$ minima.
\label{prop:virtualization}
\end{proposition}

\begin{proof}
Call the matching edges which carry flow 2 in the optimal solution to~\eqref{eq:MaximumCostFlow} constructed in \cref{lem:zero-two} the~\emph{flow edges}. Suppose there are $m$ such edges and consider their $2m$ incident vertices $v_{i_1}, \dots, v_{i_{2m}}$, which we assume to be in order among the cycle edges (recall~\autoref{def:tension graph} and also see \autoref{fig:819}). We label each of the $v_{i_1}, \dots, v_{i_{2m}}$ as ``incoming'' if they are the head of their (unique) flow edge and ``outgoing'' if they are the tail of their (unique) flow edge. Our goal is to show that the incoming and outgoing $v_{i_j}$ alternate as we go around the cycle edges.

We prove this by contradiction. With respect to the cyclic order on this list, suppose that two consecutive vertices have incoming flow edges. Without loss of generality, assume they are $v_{i_1}$ and $v_{i_2}$. First, we claim that the edges $v_{i_1} \rightarrow w_{i_1}$ and $v_{i_2} \rightarrow w_{i_2-1}$ both carry flow $+2$. Since $v_{i_1}$ receives $2$ along its flow edge and has supply $2$, its two outgoing cycle edges carry a total of $4$; each carries $0$ or $2$ by \cref{lem:zero-two}, so both carry exactly $2$, and likewise for $v_{i_2}$. Thus, they deliver $+4$ to the portion of the cycle between them: $w_{i_1}, v_{i_1+1}, \dotsc, v_{i_2-1}, w_{i_2-1}$. On the other hand, this region can only absorb $+2$, since there is one more $w$ vertex than $v$ vertices. This is a contradiction. (Remember, we have assumed that no matching edges carry flow into or out of this region.)

A similar argument shows that adjacent vertices in this set cannot both have outgoing flow edges. It follows that the incoming and outgoing flow edges alternate, as desired. Since the direction of flow goes from underarc to overarc, this means that if we follow the diagram~$D$ around, ignoring crossings which correspond to matching edges with no flow, we must alternate between overarcs and underarcs. Of course, this is the definition of a virtually alternating diagram. 

However, we may also observe that if we choose any collection of $m$ matching edges with this alternation property, we may construct a corresponding solution to the flow problem by routing flow $2$ across these matching edges (and distributing flow on the cycle edges accordingly). Therefore, the optimal solution to~\eqref{eq:MaximumCostFlow} must be given by a maximum size such subset of crossings of~$D$. But that is the complement of a minimum size subset of crossings we need to virtualize.
It follows that $m=n-k$, the optimal value of \eqref{eq:MaximumCostFlow} is $2(n-k)$, and by strong duality together with \cref{lem:mct solves fundamental}, every optimal solution of \eqref{eq:fundamental} has $\TV(f) = 2(n-k)$.

We can learn more about the structure of the optimal height function $f$ from the optimal flow function $g$ using complementary slackness. Suppose that $v_{i_j}$ is incoming, and consider the cycle edges between $v_{i_j}$ and the (outgoing) $v_{i_{j+1}}$, 
\[
v_{i_j} \rightarrow w_{i_j} \leftarrow v_{i_j+1} \rightarrow \cdots \leftarrow v_{i_{j+1}}.
\]
The edges $v_{i_j} \rightarrow w_{i_j}, v_{i_j+1} \rightarrow w_{i_j+1}, \dotsc, v_{i_{j+1}-1} \rightarrow w_{i_{j+1}-1}$ all carry nonzero flow (in fact, they all carry flow $+2$). By complementary slackness, a nonzero dual variable implies an exactly satisfied primal constraint, so we know that the corresponding primal inequalities are actually equalities:
\[
f(v_{i_j}) = f(w_{i_j}), f(v_{i_j+1}) = f(w_{i_j+1}), \dotsc, f(v_{i_{j+1}-1}) = f(w_{i_{j+1}-1}). 
\]
On the other hand, the interleaved cycle edges $w_{i_j} \leftarrow v_{i_j+1}, w_{i_j+1} \leftarrow v_{i_j+2}, \dotsc, w_{i_{j+1}-1} \leftarrow v_{i_{j+1}}$ all carry zero flow. Here, we know only that the primal constraints apply (they may or may not have slack), leaving us with $f(w_{i_j}) \geq f(v_{i_j+1}), f(w_{i_j+1}) \geq f(v_{i_j+2}), \dotsc, f(w_{i_{j+1}-1}) \geq f(v_{i_{j+1}})$. Combined, these inequalities tell us that $f(v_{i_j}) \geq f(v_{i_j+1}) \geq \cdots \geq f(v_{i_{j+1}})$, or that $f$ is nonincreasing between $v_{i_j}$ and $v_{i_{j+1}}$. Symmetrically, if $v_{i_j}$ had been outgoing, we would have learned that $f$ was nondecreasing between $v_{i_j}$ and $v_{i_{j+1}}$. That is, an incoming $v_{i_j}$ is a (weak) local maximum for $f$ and an outgoing $v_{i_j}$ is a (weak) local minimum. In particular, there are at most $n-k$ local minima and $n-k$ local maxima.
\end{proof}

\begin{figure}
\hphantom{.}
\hfill
\raisebox{-0.5\height}{\includegraphics[width=0.2\textwidth]{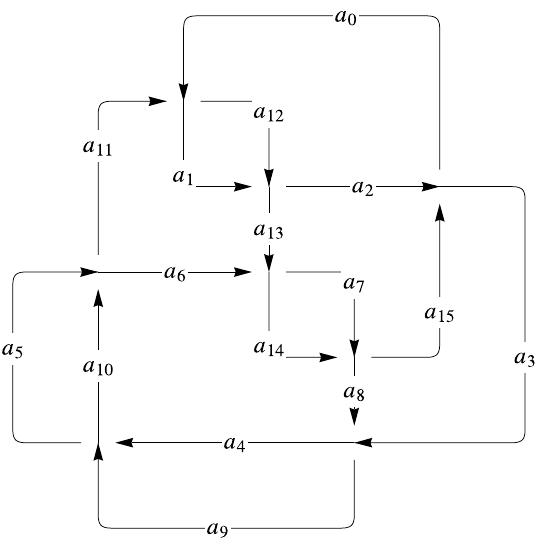}}
\hfill
\raisebox{-0.5\height}{\includegraphics[width=0.2\textwidth]{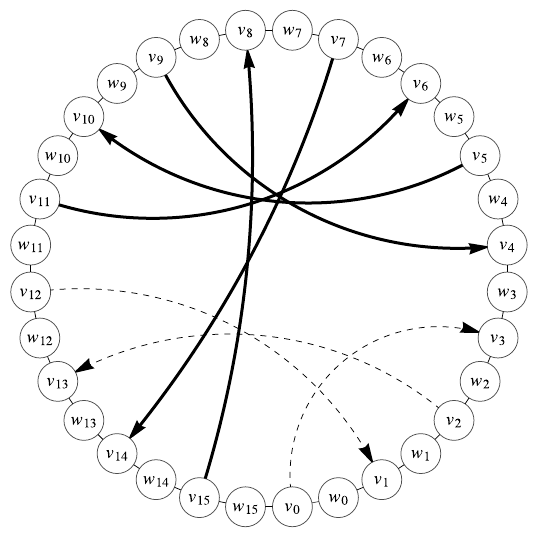}}
\hfill
\raisebox{-0.5\height}{\includegraphics[width=0.2\textwidth]{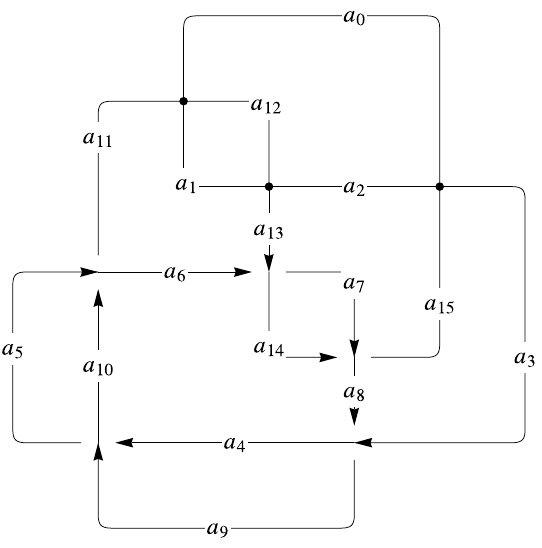}}
\hfill
\raisebox{-0.5\height}{\includegraphics[width=0.25\textwidth]{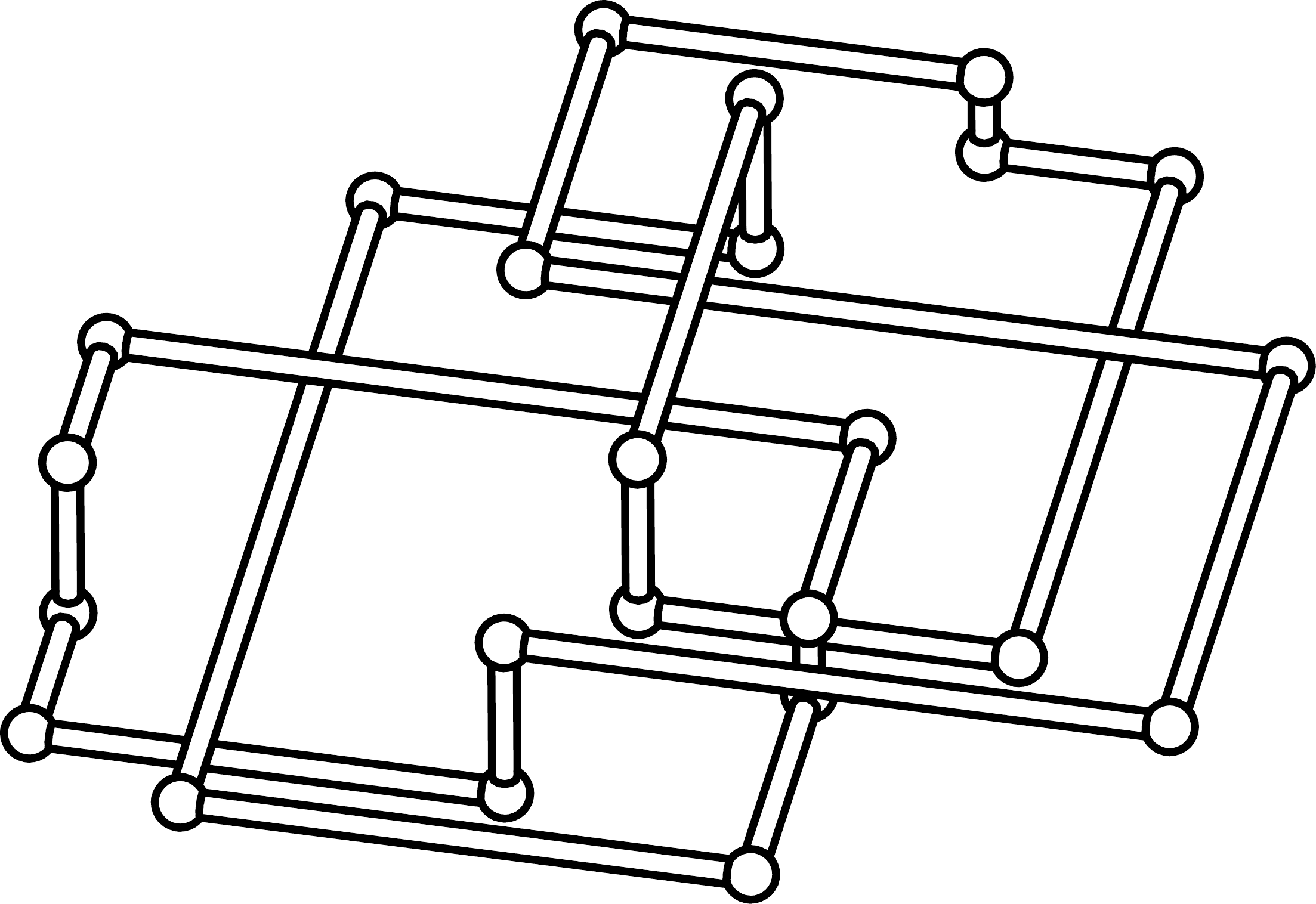}}
\hphantom{.}
\caption{On the left, a nonalternating diagram with $n=8$ crossings. This is the standard diagram for the $8_{19}$ knot. We have labeled the arcs $a_i$. In the center, we see the corresponding tension graph $T(D)$, with vertices $v_i$ corresponding to the arcs $a_i$ with vertices $w_i$ inserted between them. The $k=3$ matching edges which do not carry flow are dashed, while the $n-k = 5$ matching edges which carry flow $2$ are drawn in black. The maximum cost is therefore $10$. In the third panel, we see that if we virtualize the $3$ crossings corresponding to the dashed edges, the new knot is virtually alternating (that is, it is alternating if we ignore virtual crossings). This must be a smallest set of crossings which makes $D$ virtually alternating. On the far right, we see the corresponding lattice embedding, which has $3$ local maximum arcs; this is less than the maximum of $n-k=5$.}
\label{fig:819}
\end{figure}

Solving~\eqref{eq:fundamental} as a generic $\ell^1$ optimization problem with linear constraints using standard LP solvers was by far the slowest step in our initial implementation of the \Reapr algorithm. Rephrasing it as the MCF problem~\eqref{eq:MaximumCostFlow} was a very significant optimization. Using the network simplex algorithm from~\cite{MCFClass} reduces the computational cost to something comparable to the orthogonal graph layout.

\section{Test set and Performance Comparisons}
\label{sec:effectiveness}

We tested our algorithm on approximately 2.6 million unknot diagrams with between 5 and 130,155 crossings, drawn from 10 test sets which we have deposited on DataDryad~\cite{testsets}:\label{test sets}

\begin{enumerate}[label=(\Alph*)]
\item 200 random lattice unknots filling spheres, created by the BFACF~\cite{bergRandomPathsRandom1981,aragaodecarvalhoNewMonteCarloApproach1983,aragaodecarvalhoPolymers$g|varphi|^4$Theory1983} algorithm, 
\item the 18 unknots among the ``children's game'' knots of Gukov et al.~\cite{GukovLearning}, 
\item 200 self-avoiding unknotted 10,000-gons in tight spherical confinement~\cite{OliveraPC}, 
\item 5 complicated unknots of the ``hard Goeritz'' type from~\cite{BurtonHardUnknots},
\item the 21 ``hard unknot'' examples in~\cite{BurtonHardUnknots}, 
\item $27$ examples of rational tangle unknots from ``Kauffman's challenge''~\cite{kauffmanKnotInvariantsKnot2024},
\item 2,623,203 pass-reduced ``hard unknots'' found by Applebaum et al.~\cite{applebaumUnknottingNumberHard2025},
\item 600 unknotted random equilateral 3043-gons~\cite{CantarellaSchumacherShonkwiler2024}, 
\item 176 pass-reduced unknots found by Tuzun and Sikora~\cite{TuzunSikoraJones},
\item 384 pass-reduced ``very hard unknots'' from~\cite{applebaumUnknottingNumberHard2025}.
\end{enumerate}
\Reapr recognizes all of these diagrams as unknots. We compared \Reapr to diagrammatic and non-diagrammatic unknot recognition algorithms in \Regina \cite{Regina} and \SnapPy \cite{SnapPy}. 

The diagrammatic simplification method \texttt{Link.simplify()} in~\Regina alternates between random Reidemeister III moves and simplifications using RI and RII moves. \SnapPy's diagrammatic \texttt{Link.simplify(\textquotesingle global\textquotesingle)} method mixes pass moves and random RIII moves. We used the default limit of 100 random RIII moves. We used~\texttt{Link.deconnect\_summand()} to split off connected summands where possible, which sometimes revealed further simplifications.

We also used three non-diagrammatic methods, which work with a triangulation of the knot complement rather than with the diagram. For \SnapPy (``\Snappy fund.\ grp.'' in~\cref{tab:benchmarks}), we followed a procedure suggested to us by Nathan Dunfield. We try a number of random tetrahedralizations of the knot complement and then simplify the fundamental group of this tetrahedral complex, trying to eliminate all relators. Our two \Regina methods (``\Regina triang.'' and ``\Regina fund.\ grp.'') follow the same procedure with \Regina's routines: after RI and RII reductions, \texttt{Link.complement()} builds a simplified triangulation of the complement, which we rebuild and re-simplify at random, up to fixed limits, until it has fewer than 10 tetrahedra. The \Regina fundamental group method then tries to eliminate all relations from the fundamental group, as \SnapPy does, while the \Regina triangulation method calls \texttt{isSolidTorus()}, which decides whether the complement is a solid torus. The triangulation method's answer is a proof either way. The two fundamental group methods, like \Reapr itself, can only certify that a diagram is unknotted. All three methods are randomized (we reseed \Regina's random number generator for each diagram), so each of their entries in~\cref{tab:benchmarks} reports a sample. \Cref{sec:protocols} gives the source code for all three.

All timings were obtained on a single core of a 2022 Mac Studio (Apple M1 Ultra). For \Regina and \SnapPy, we set a per-knot timeout of 100 seconds for the $\approx 2.6$ million `hard unknots' in set G, 1000 seconds for set A, and 3000 seconds for all other knots, and a per-knot memory limit of 4~GB (1~GB for the \SnapPy diagrammatic method on sets A, B, D, H, I and J, which never reached it); neither limit was needed for \Reapr. The results are shown in~\cref{tab:benchmarks}.

\begin{table}
\centering
\setlength{\tabcolsep}{5pt}%
\begin{tabular}{@{}l@{\hspace{3pt}} r r l@{\hspace{2pt}}l l@{\hspace{2pt}}l l@{\hspace{2pt}}l l@{\hspace{2pt}}l l@{\hspace{2pt}}l l@{\hspace{2pt}}l@{}}
\toprule
 & & & \multicolumn{2}{c}{\Reapr} & \multicolumn{6}{c}{\Regina} & \multicolumn{4}{c@{}}{\Snappy} \\
\cmidrule(lr){4-5} \cmidrule(lr){6-11} \cmidrule(l){12-15}
set & knots & cr. &  & & \multicolumn{2}{c}{diagram} & \multicolumn{2}{c}{triang.} & \multicolumn{2}{c}{fund.\ grp.} & \multicolumn{2}{c}{diagram} & \multicolumn{2}{c@{}}{fund.\ grp.} \\
\midrule
A & 200 & 111K & \textbf{8.65s} & \textbf{1.0} & 36h & .87 & 56h & .00$^\dag$ & 56h & .00$^\dag$ & 56h & .00$^\dag$ & 9.30m & \textbf{1.0} \\
B & 18 & 18 & .24 & \textbf{1.0} & \textbf{.07} & \textbf{1.0} & 8.22 & \textbf{1.0} & 5.34 & \textbf{1.0} & 4.30 & \textbf{1.0} & 5.90 & \textbf{1.0} \\
C & 200 & 18K & \textbf{1.08s} & \textbf{1.0} & 18m & .93 & 54h & \textbf{1.0} & 54h & \textbf{1.0} & 72h & \textbf{1.0} & 70.6s & \textbf{1.0} \\
D & 5 & 85 & 20.8 & \textbf{1.0} & 5.11 & .20 & 251 & \textbf{1.0} & 214 & \textbf{1.0} & 1.5s & .60 & \textbf{6.30} & \textbf{1.0} \\
E & 21 & 45 & \textbf{24.7} & \textbf{1.0} & 4.78 & .24 & 10s & \textbf{1.0} & 2.33s & \textbf{1.0} & 2.7s & .86 & 109 & \textbf{1.0} \\
F & 27 & 50 & \textbf{0.99} & \textbf{1.0} & .97 & .15 & 11s & .93 & 30m & \textbf{1.0} & 5.6s & \textbf{1.0} & 680 & \textbf{1.0} \\
G & $\approx$ 2.623m & 22 & \textbf{17.57s} & \textbf{1.0} & 50.1s & .00$^*$ & 4.6h & 1.0$^*$ & 1.1h & \textbf{1.0} & 15h & 1.0$^*$ & 16.4m & \textbf{1.0} \\
H & 600 & 3.3K & \textbf{167} & \textbf{1.0} & 6.46s & .98 & 39m & \textbf{1.0} & 39m & \textbf{1.0} & 20m & \textbf{1.0} & 31s & \textbf{1.0} \\
I & 176 & 15 & 38.4 & \textbf{1.0} & \textbf{1.40} & \textbf{1.0} & 163 & \textbf{1.0} & 99.1 & \textbf{1.0} & 259 & \textbf{1.0} & 46.8 & \textbf{1.0} \\
J & 384 & 28 & \textbf{122} & \textbf{1.0} & 8.44 & .00 & 1.29s & \textbf{1.0} & 1.15s & \textbf{1.0} & 23s & .02 & 162 & \textbf{1.0} \\
\midrule
 & $\approx$ 2.625m & 33 & \textbf{27.67s} & \textbf{1.0} & 36h & .00$^*$ & 115h & 1.0$^*$ & 112h & 1.0$^*$ & 143h & 1.0$^*$ & 27.4m & \textbf{1.0} \\
\bottomrule
\end{tabular}
\caption{Comparison of unknot recognition methods across ten test sets (described on page~\pageref{test sets}; total: 2,624,834 diagrams). The ``cr.''\ column gives the average crossing number. Each entry shows total running time followed by effectiveness (fraction of unknots identified). Times are in milliseconds unless marked with ``s'', ``m'', or ``h'' (for ``seconds,'' ``minutes,'' or ``hours,'' respectively). Effectiveness \textbf{1.0} = all unknots recognized; $1.0^*$ rounds to~1.0; $.00^*$ rounds to~0. A dagger indicates that on all samples \Regina or \Snappy exceeded time or memory constraints before terminating. Boldface marks the fastest method achieving perfect effectiveness.}
\label{tab:benchmarks}
\end{table}

\Reapr and the \snappy fundamental group method both achieved perfect effectiveness on every test set, and the \Regina fundamental group method did so on every test set except~A (discussed below). The \Regina triangulation method was also generally effective, though there were 110 pass-reduced examples in~G where it timed out at 100 seconds. (These timeouts consumed about two thirds of its runtime on test set~G.) The two diagrammatic methods in \Regina and \SnapPy were less effective: \Regina's \texttt{Link.simplify()} made almost no progress on test sets~G and~J (though it is extremely fast on small diagrams), while \SnapPy's diagrammatic method failed on the larger pass-reduced examples in~J.

\Reapr and the \Snappy fundamental group method were the only methods which scaled well to large diagrams. Both \Regina non-diagrammatic methods and the \SnapPy diagrammatic method ran out of time on all members of test set~A (200 lattice unknots, average $\approx 111$K crossings). The two \Regina methods timed out while building their first triangulation of the knot complement. In a spot check on one of these knots, that step took about 30 hours (and produced a triangulation with just 2 tetrahedra), so we expect that both \Regina methods would eventually have succeeded. We project that each would have needed $\approx 0.7$ years of CPU time for the whole set. A similar spot check suggests that the \SnapPy diagrammatic method would have needed $\approx 1$ year. \Regina's diagrammatic simplifier was reasonably effective on test set A, but took many hours. By contrast, \Snappy's fundamental group method processes the same set in under 10 minutes and \Reapr processes it in under 10 seconds. Scaling to large diagrams is a significant practical consideration: in~\cite{RandomKnots}, we needed to simplify diagrams with up to 35 million crossings.

On test set~F, the Kauffman challenge unknots~\cite{kauffmanKnotInvariantsKnot2024}, \Regina's triangulation method could not recognize two of the 27 diagrams, and its fundamental group method recognized all of them but took 30 minutes, against under a second for \Snappy's. For the four largest diagrams in the set (113 to 175 crossings), the first triangulations of the knot complement built by \Regina and by \Snappy usually have between 100 and 200 tetrahedra, roughly one per crossing, and rebuilding them helps less and less as the diagrams grow. In 20 trials on each of these four diagrams, \Snappy's \texttt{randomize()} brought every triangulation down to between 2 and 7 tetrahedra within at most 73 calls, while 100 calls of \Regina's randomized \texttt{simplify()} barely reduced it: for the largest diagram, after 11 builds and 100 calls, the triangulation still had between 161 and 185 tetrahedra in every trial. On triangulations this large, \texttt{isSolidTorus()} exceeded 4~GB of memory within seconds, though the fundamental group simplifier was eventually able to remove all the relations. In the other direction, Kauffman's design for the Set F unknots allows them to be simplified monotonically by pass moves and they are therefore straightforward for the pass-move based \Snappy diagrammatic method and \Reapr.

There are many other simplification strategies described in the literature.
Petronio and Zanellati~\cite{PZmoves} use nonlocal moves which include pass moves and generalized versions of Reidemeister I and II moves which avoid other arcs crossing over or under a region to be simplified, as well as moves dividing diagrams into connect sums. Andreeva et al.~\cite{andreevaMathematicalWebserviceRecognizing2002} used an arc presentation to construct an interactive simplification applet in Java.  Balch~\cite{BalchPCB} used a combination of pass moves and some ``hard'' moves roughly similar to the Petronio–Zanellati moves. We did not test these methods because none of their implementations seem to be actively maintained.

Another very different approach uses the flow of a repulsive energy to unknot curves.  These methods start with a geometric realization of the knot as a curve in space and use numerical optimization methods to try to minimize the energy. Here the state of the art is Noma~et~al.~\cite{nomaMedialSpherePreconditioning2025} which requires several seconds to untangle the 32 crossing Freedman–He–Wang unknot, faster than the method proposed by one of us (Schumacher)~\cite{Yu2021}, but extremely slow compared to the tests above. These methods are often defeated by test set~F (the Kauffman challenge unknots~\cite{kauffmanKnotInvariantsKnot2024}). We did not test them here.

\section{Conclusions and Future Directions}

The most pressing mathematical question about \Reapr is the simplest: why does it work? Petronio and Zanellati~\cite{PZmoves} noticed that even pass-reduced knot diagrams may yet contain crossing-number reducing passes which reroute ``mid-level'' arcs of a knot diagram, as in~\cref{fig:middle move}. \Reapr can certainly discover such moves by providing a new projection in which mid-level arcs become over- or underpasses. Preliminary experiments suggest that this accounts for most, but not all, of \Reapr's simplifications: there appear to be cases requiring coordinated rearrangement of larger portions of the diagram. While \Reapr appears to appeal to 3d information about a lattice embedding, this embedding is computed from the diagram --- in particular, the minimum total variation is a combinatorial invariant of the diagram (\cref{prop:virtualization}).

\begin{figure}[t]
\centering
\hfill
\includegraphics[width=0.17\textwidth]{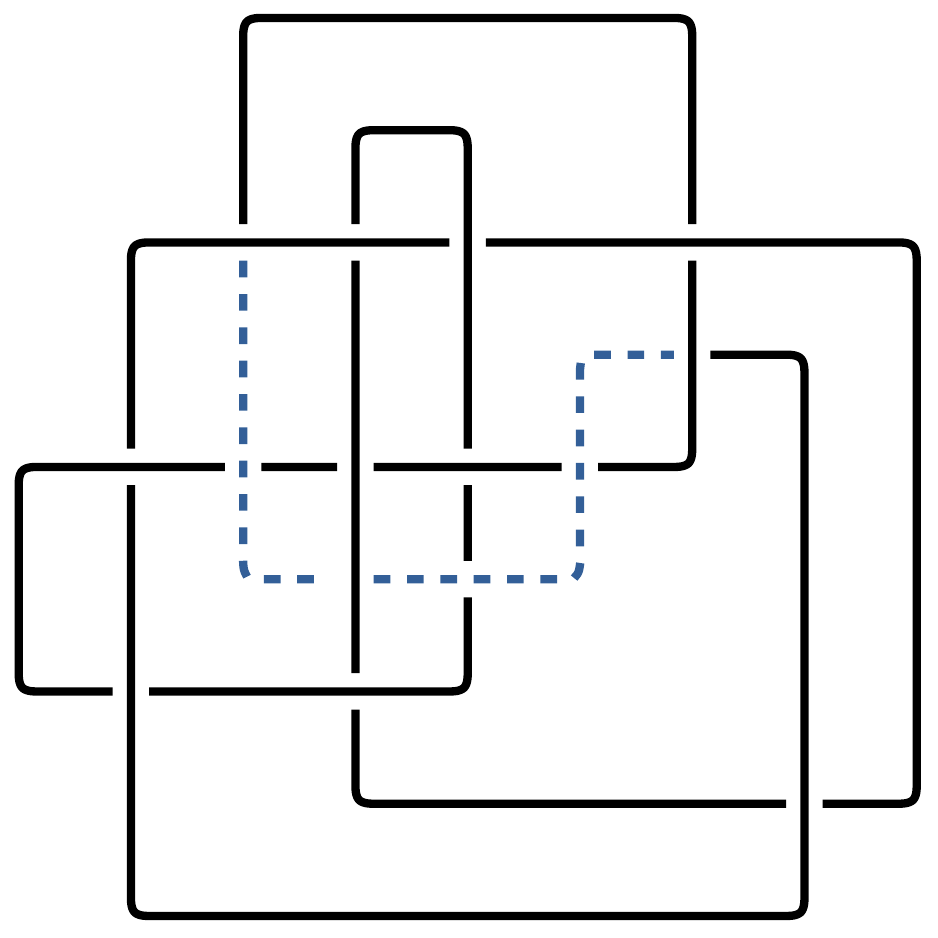}
\hfill
\includegraphics[width=0.17\textwidth]{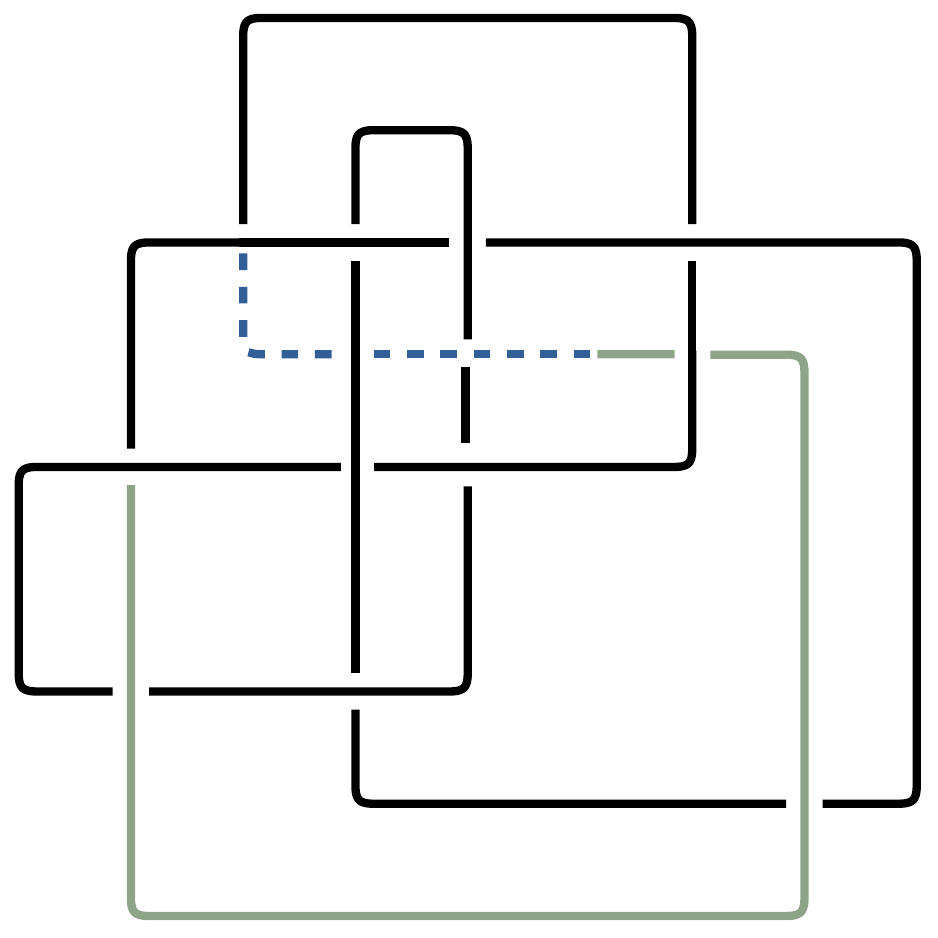}
\hfill
\includegraphics[width=0.17\textwidth]{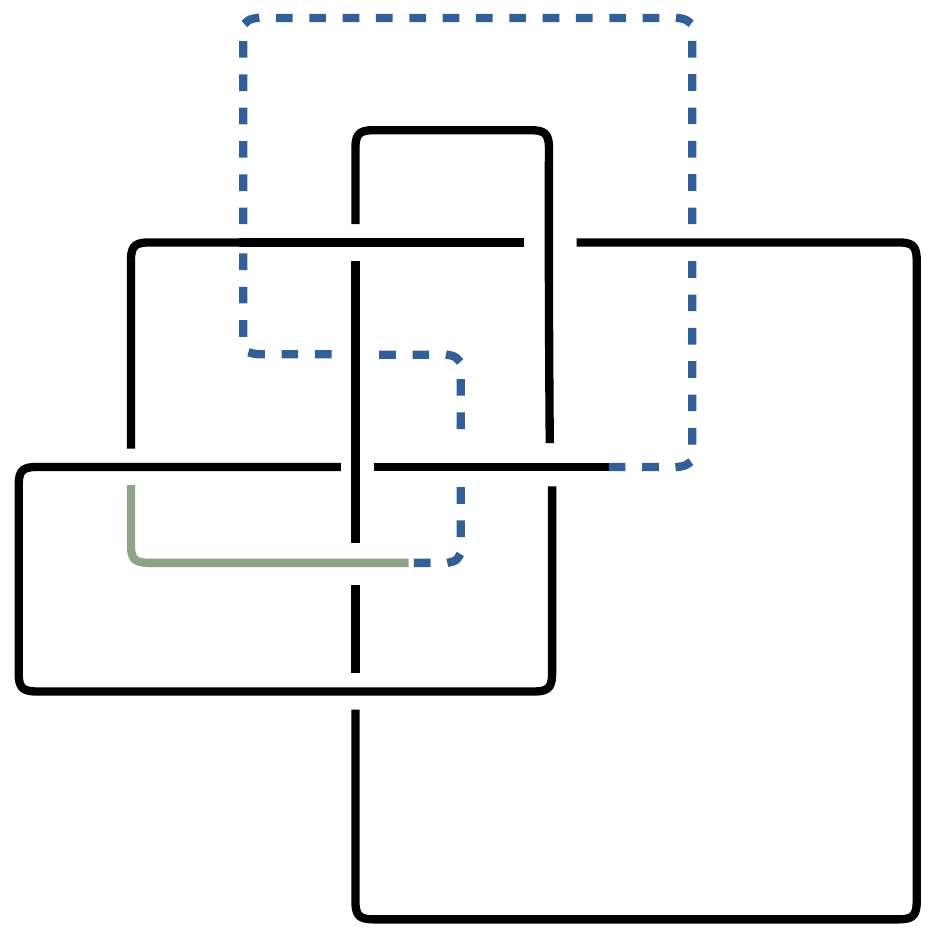}
\hfill
\includegraphics[width=0.17\textwidth]{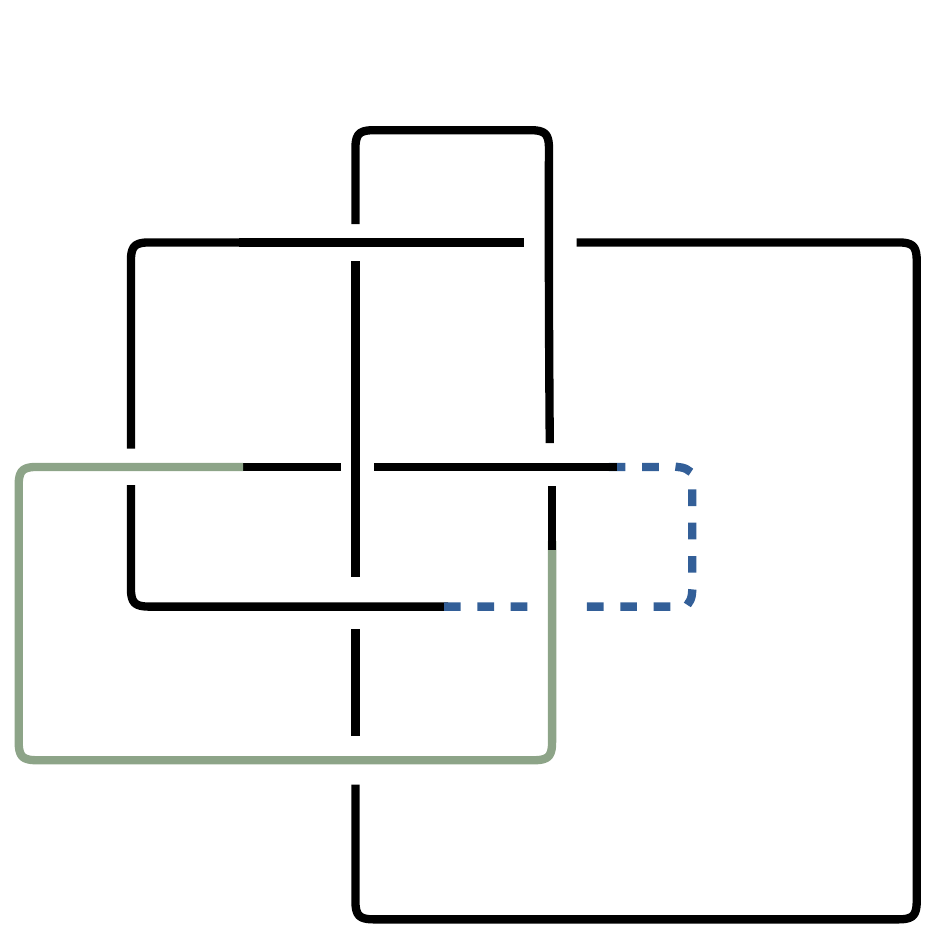}
\hfill
\includegraphics[width=0.17\textwidth]{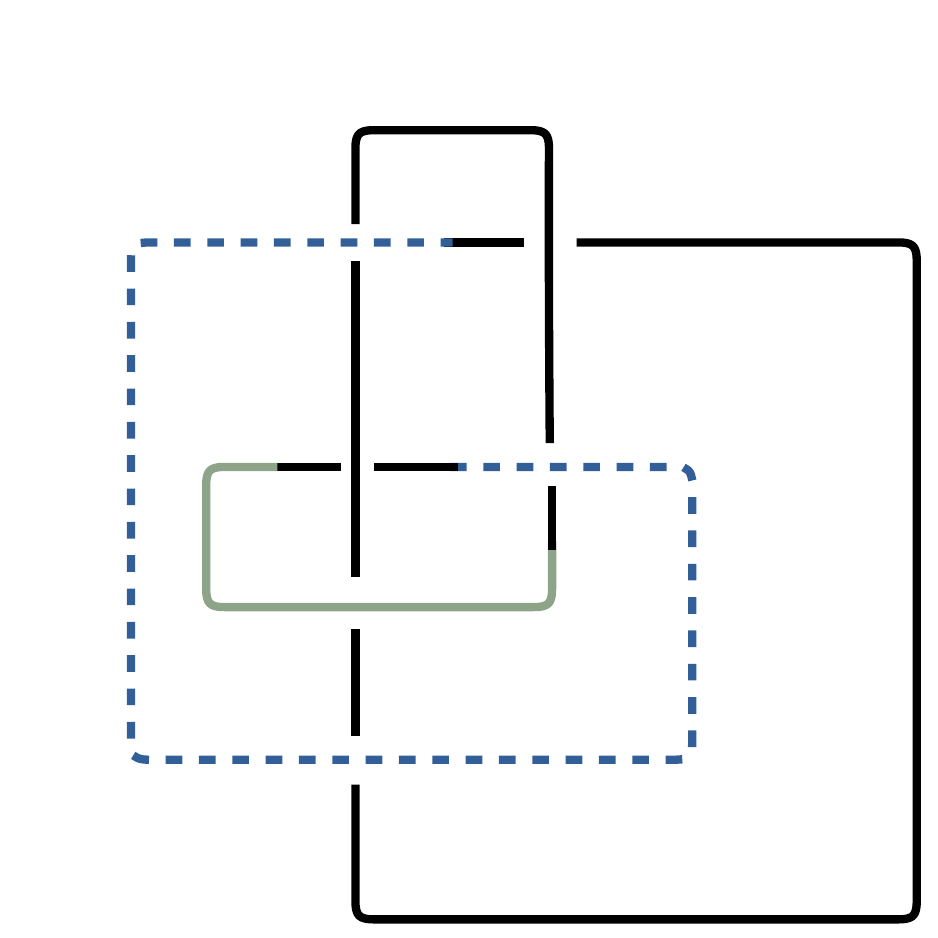}
\hfill
\hphantom{.}
\caption{An illustration of ``mid-level'' crossing-number reducing pass moves in the 146th pass-reduced unknot discovered by Tuzun and Sikora~\cite{TuzunSikoraJones}. The first two moves are of this type; after that, the diagram is reduced by conventional pass moves (the last few of which are left to the reader). \Reapr discovers mid-level pass moves empirically by providing a ``side view'' into the knot diagram.
\label{fig:middle move}}
\end{figure}

The effectiveness of pass moves at making large-scale changes to diagrams suggests that adding random crossing-number preserving pass moves might be a low-cost way to significantly improve the performance of diagram simplifiers (including \Reapr itself!). 

One important design goal of \Reapr was to be as fast and efficient as possible. As a result, \Reapr may be a useful tool in the training pipeline for AI methods in knot theory. For example, Applebaum et al.~\cite{applebaumUnknottingNumberHard2025} conclude their paper with an example of a 42-crossing hard unknot which resisted $10^4$ attempts to simplify it using \Snappy's \texttt{\textquotesingle global\textquotesingle} heuristic; \Reapr simplifies this example in 0.36~ms on a 2025 laptop. %

\section*{Acknowledgments}
We are very grateful to our colleagues and friends for many helpful discussions about knots and unknots and their generous sharing of data, including Ben Burton, Ryan Budney, Nathan Dunfield, Ronaldo Oliveira, Yo'av Rieck, Adam Sikora, and Lou Kauffman. Some of this work was done as part of the AIM SQuaRE ``Landscapes of Knots'', and we are grateful to the American Institute of Mathematics for their generous support. 
This work was partially supported by the National Science Foundation (DMS--2107700).
H.S.{} gratefully acknowledges support through the research training group ``Energy, Entropy, and Dissipative Dynamics'', funded by Deutsche Forschungsgemeinschaft, project no. 320021702/GRK2326. Claude Code was used to write scripts to drive experiments and organize data, as well as to write tests for~\Knoodle and~\Reapr. The core algorithms of \Knoodle and \Reapr are coded by H.S..

\bibliography{extra-custom-citations,zotero-library-current}

\appendix

\crefalias{section}{appendix}
\crefalias{subsection}{subappendix}
\crefalias{subsubsection}{subsubappendix}

\section{Pass Rerouting Details}
\label{sec:rerouting details}

\subsection{Algorithmic Details}
\label{sub:AlgorithmicDetails}

We provide here more details on the pass rerouting algorithm described in~\cref{sec:pass moves}. Pass rerouting consists of three steps:
\begin{enumerate}
    \item Finding a (maximal) over- or underpass.
    \item Find a good target path to reroute to.
    \item Performing the actual rerouting on the diagram.
\end{enumerate}

We now explain each in some detail.

\subsubsection{Finding a Maximal Pass}
\label{sec:FindingPasses}

To find an over-/underpass in the link diagram $D$, we start at some active arc $a_0$.
If its head goes over/under, we start searching an over-/underpass.
For simplicity, let us assume that we are dealing with an overpass. (Underpasses are processed analogously.)

If its tail goes under, then we can be sure that this arc is the start of a maximal overpass.
Otherwise, we are somewhere in the middle of a maximal overpass; then we walk backwards to its start.
Once the first arc $a_1$ of the maximal overpass is found, we walk along the link, traversing arcs $a_1,a_2,\dotsc$ until we visit an arc $a_k$ whose head goes under.
Then $(a_1,\dotsc,a_k)$ is a (maximal) overpass.
For later use we mark these arcs and all their crossings $c_0,\dotsc,c_k$.
Since we have to reroute many passes, we do this in a way that allows us to unmark all arcs and crossings in an instant:
For each crossing and each arc we set aside an integer for its mark and initialize these integers to~$0$.
Whenever we start processing a new pass, we first increment an internal counter; 
as soon as we visit crossing $c_i$ and arc $a_i$, we set their marks to the current value of this pass counter.
This identifies them as members of this pass. (This way we also do not have to store the tuples $(a_1,\dotsc,a_k)$ and $(c_0,\dotsc,c_k)$ explicitly.)
As soon as we start processing another pass, we increment the pass counter again, immediately rendering the marks on all arcs and crossings stale. 
So, in this sense, unmarking arcs and crossings costs $O(1)$.

\subsubsection{Finding Shortest Paths}

Let $D$ be our current link diagram.
To find a good target path to reroute to, we consider the new diagram $\hat{D}$ obtained from $D$ by removing the marked arcs.
Moreover, we let $\hat{D}^*$  be the \emph{dual diagram} of $\hat{D}$:
Its vertices are the faces of $\hat{D}$; its arcs correspond to the arcs of $\hat{D}$; and its faces correspond the crossings of $\hat{D}$ (but the latter do not really matter here).

Let $(a_1,\dotsc,a_k)$ be the pass we want to reroute.
The two faces of $a_1$ in $D$ collapse to one face $f_1$ in $\hat{D}$, which gives rise to a vertex in $\hat{D}^*$.
Likewise, the two faces of $a_k$ collapse to a face $f_k$ in $\hat{D}$ and thus to a vertex in $\hat{D}^*$.
Our goal is to find a shortest path of dual arcs $b_1^*,b_2^*,\dotsc,b_\ell^*$ in $\hat{D}^*$ from $f_1$ to $f_k$. 
Each of the dual arcs $b_i^*$ means that we have to cross their corresponding primal arc $b_i$ in $\hat{D}$, creating a new crossing. 
Such a rerouting allows us to replace the $k-1$ crossings $c_1,\dotsc,c_{k-1}$ of the overpass by just $\ell$ crossings, provided that this shortest path consists of $\ell <  k - 1$ dual arcs. 
Hence, minimizing the path length in $\hat{D}^*$ means minimizing the number of crossings after rerouting,
so what we are looking for is a shortest path in $\hat{D}^*$.

Creating a new knot diagram for $\hat{D}$ and creating a new dual diagram $\hat{D}^*$ costs $O(n)$. As explained in~\cref{sec:pass moves},
we will instead navigate in $\hat{D}^*$ using the operators $\Left(\cdot)$ and $\Reverse(\cdot)$ on darcs. 

The key difference in navigating on $\hat{D}^*$ rather than $D^*$ is in modifying the $\Left(\cdot)$ operator to ignore deleted arcs, which we do using the marks on the primal arcs.
Specifically, the only difference from traversing $D^*$ is in what we do when we happen to visit a darc $da$ whose underlying arc is marked as member of the current overpass.
In this case, we simply ignore $da$ and move on to  $da \leftarrow \Left (\Reverse(da))$.
This requires us to maintain a relationship between 
an arc $a$ and its forward and backward darcs $da_{\mathrm{forward}}$ and $da_{\mathrm{backward}}$.
For example, one can set $da_{\mathrm{forward}} = 2 \, a + 1$ and $da_{\mathrm{backward}} = 2 \, a + 0$.
This way, finding the arc underlying a darc amounts to a bit shift; the direction can be inferred by reading off the least significant bit;
and computing $\Reverse(da)$ amounts to flipping the least significant bit of the integer representing $da$. 
So the mapping between arcs and darcs can be done entirely without memory lookups.

The remainder of the pathfinding algorithm can be summarized as bidirectional Dijkstra search in $\hat{D}^*$ with unit edge weights on dual arcs.
As the implementation details matter, we provide at least the most important ones.
The idea of the bidirectional Dijkstra search is to grow two balls around the two starting vertices $f_1$ and $f_k$ in $\hat{D}^*$. Let us call these balls $A$ and $B$. We start with radii $r_A = 0$ and $r_B = 0$. 
Then one ball, say $A$, is selected, and its radius is increased by one unit. 
This is done by visiting the \emph{boundary} of $A$ and then visiting the neighbors of the boundary elements from there.
If such a neighbor is already in $A$, we do nothing. 
If it is not in $A$ we add it to $A$---unless it is already a member of the other ball $B$. 
In this case, the two balls $A$ and $B$ touch for the first time; then our search has come to an end.
A shortest path can be reconstructed by maintaining for each visited vertex the vertex from which it was first visited.

Typically, this algorithm is described by focusing on the vertices.
Since the vertices of interest are faces of the modified diagram $\hat{D}$ in our case, and since we do not even want to assign a name to them, we reformulate the usual bidirectional algorithm focusing on directed arcs of $\hat{D}$, which we identify with the subset of those darcs of $D$ whose underlying arc is unmarked.
For this we use two stacks $A_{\mathrm{front}}$ and $B_{\mathrm{front}}$ to maintain the lists of those darcs $da$ whose duals ``stick out''  of $A$ and $B$ respectively.
More precisely, a darc $da$ is a member of $A_{\mathrm{front}}$ if and only if its right face ($=$ the tail of $da^*$) is a member of $A$ and if its left face ($=$ the tip of $da^*$) belongs to neither $A$ nor $B$.
Moreover, we store the following information for each arc $e$ of~$D$:
\begin{itemize}
    \item whether $e$ has already been visited in the bidirectional search for the current overpass; this is done by storing the current value of the overpass counter;
    \item the arc underlying the darc that lead to $e$ being visited;
    \item whether $e$ was visited while growing $A$ or while growing $B$;
    \item the direction in which $e$ was crossed upon that visit (``left-to-right'' or ``right-to-left'' with regard to the forward arc $de_{\mathrm{forward}}$).
\end{itemize}
The first piece of information (the mark) allows us instant reset of information by the mechanism described above. 
The other data is used for the actual search and for the reconstruction of the optimal path through backtracking.
At start-up of the search, the stacks $A_{\mathrm{front}}$ and $B_{\mathrm{front}}$ are empty and all arcs are implicitly marked as unvisited because their marks are stale.
Then we cycle through the darcs of the faces $f_0$ and $f_k$ (in the diagram $\hat{D}$) by the mechanism described above.
For each darc $da$ visited this way, we put the data described above onto the underlying arc $a$, and push its \emph{reversal} $\Reverse(da)$ to $A_{\mathrm{front}}$ and $B_{\mathrm{front}}$, respectively.
Now the darcs on the stacks point to the faces we have to visit next for growing the balls $A$ and $B$, i.e., these faces are to the left of these darcs.

Next we describe how to grow one of the balls, say $A$.
First we provide an empty stack $F$ to hold the new frontier of $A$.
Then we pop a darc $da$ from the stack $A_{\mathrm{front}}$.
Let us temporarily denote the face that lies to its left by $f$ and the underlying  arc by $a$.
Now we use the cycling method described above to visit all the other darcs of the face $f$. 
For each darc $de$ visited this way,
we check whether and how the underlying arc $e$ has been visited already:
\begin{itemize}
    \item If $e$ has not been visited from either $A$ or $B$, we push $\Reverse(de)$ to the stack $F$ 
    and store the following information in $e$: 
    \begin{itemize}
        \item the current value of the overpass counter (to remember that we have visited $e$); 
        \item that $e$ was visited from $a$; 
        \item that $e$ was visited while growing $A$; and
        \item the direction in which $e$ was crossed just now: 
        ``left-to-right'' if $de = 2\, e +1$ or ``right-to-left'' if $de = 2\,e +0$.\footnote{Note that this flag has to be assigned differently when we are about to grow $B$ because the according part of the path will go backwards: here $de = 2\, e + 0$ means ``left-to-right'' and  $de = 2\, e + 1$ means ``right-to-left''. }
    \end{itemize}    
    \item If $e$ has been visited while growing $A$ in the same direction, then we stop cycling around~$f$ because we know that all other arcs have been visited, too.
    \item If $e$ has been visited while growing $A$, but in the opposite direction, then we know that the face on the other side has been visited already. Hence, we do nothing and continue cycling over the darcs of $f$.
    \item If $e$ has been visited while growing $B$, then the balls $A$ and $B$ just touch for the first time; we stop growing $A$ and $B$ immediately and move on to reconstructing the shortest path via backtracking.
\end{itemize}

We do this for all darcs on stack $A_{\mathrm{front}}$ until the latter is empty. Then we swap $A_{\mathrm{front}}$ with $F$ so that $A_{\mathrm{front}}$ now represents the new frontier. (The stack $F$ can be reused for the next frontier.)

All that is needed now is a strategy for deciding whether to grow $A$ or $B$.
Each such strategy leads to a different variant:
always growing $A$ leads to the original Dijkstra algorithm (in the special case of unit edge weights), also known as \emph{unidirectional breadth-first search};
growing $A$ and $B$ in alternating fashion would resemble what is often  referred to as \emph{bidirectional breadth-first search}.
The bidirectional search is more efficient because the  main cost of growing a ball in a graph is to visit all the edges on the boundary.

However, as described in~\cref{sec:pass moves}, we found in practice that another strategy works better: always grow the ball with the smaller frontier (or just the first one if the two frontiers have the same size). 
The size of the frontier is some indicator on how many edges are to be visited in the next sweep.\footnote{The exact number depends on how many of the frontier's neighbor faces are unvisited and on the size of these unvisited faces. Alas, this information is not immediately available at this moment; so we have to use the size of the frontier as a proxy for this.}
So, this is a greedy strategy to minimize the total number of edges ($=$ dual arcs) visited.

\subsubsection{Rerouting}

This rerouting is somewhat tedious:
The interior crossings $c_1,\dotsc,c_{k-1}$ of the pass have to be disconnected from the arcs that do not belong to the pass.
For each crossing disconnected this way, one of the incident arcs has to be deleted, and another has to be reconnected to where the other arc was connected---unless these arcs coincide; then the arc is deleted and an unlink is recorded. (This can happen only in diagrams for multiple-component links.)
Let $b_1,\dotsc,b_\ell$ be the arcs crossed by the shortest path from $f_1$ to $f_k$. 
These arcs need to be split and a new crossing and two new arcs have to be sewn in appropriately. 
For the new crossings we can reuse the memory of crossings $c_1,\dotsc, c_\ell$ because we have the guarantee that $\ell < k - 1$; the remaining crossings are ``deleted'' by setting their state flags to ``inactive'' (see \cref{sec:DataLayout} for details).
Similarly, we can reuse the memory of the deleted arcs for the newly created arcs. 

This all is straightforward, provided that neither the pass nor the target path contain any loops (except maybe $c_0 = c_k$).
The target path cannot contain any loops because we chose it as a shortest path.
But the pass may contain loops, in principle.
Preventing this was the reason we put marks on the crossings during detection of passes (see \cref{sec:FindingPasses}): once we revisit an already marked crossing, 
we know that the current strand forms a loop. 
We immediately stop expanding the pass and remove the loop by the surgery described above.

\subsection{Implementation Details}
\label{sec:implementation}
\subsubsection{Data Layout}
\label{sec:DataLayout}

Memory layout is crucial for the performance of our algorithm.
The rerouting algorithm frequently alters the diagram. So it is tempting to design the data structure as a~linked list: a crossing could be represented by a quadruple of pointers to the incident arcs; an arc could be represented by a pair of pointers to the crossings; and instances of crossings and arcs could be created (i.e., allocated) and deleted (i.e., deallocated) on the fly.
This provides maximal flexibility, but it gives little control on where each crossing or arc is located in memory (which might lead to many expensive cache misses); and the overhead of frequent (de-)allocation may be considerable.

This is why we choose a different layout for the class \texttt{PlanarDiagram} in \emph{Knoodle}.
Rerouting always reduces the number of crossings and the number of arcs. 
Hence, we decided to make deletion fast and to entirely neglect creation of crossings and arcs.
We store all $n$ crossings in a $3$-dimensional heap-allocated array \texttt{\CA} (read: \underline{c}rossings to \underline{a}rcs) of size $n \times 2 \times 2$ so that the index of an arc adjacent to crossing \texttt{c} can be accessed via \texttt{\CA[c][i][j]}, where \texttt{i} can be \texttt{0} (outgoing) or \texttt{1} (incoming), and \texttt{j} can be \texttt{0} (left) or \texttt{1} (right).
In the layout convention of the \emph{C}~language this means that 
\texttt{\CA[c][0][0]}, \texttt{\CA[c][0][1]}, \texttt{\CA[c][1][0]}, \texttt{\CA[c][1][1]}
are stored consecutively (and in this order).
Likewise, we store the $m = 2 \, n$ arcs in a $2$-dimension array \texttt{\AC} (read: \underline{a}rcs to \underline{c}rossings) of size $m \times 2$ so that the indices of two crossing of arc \texttt{a} are located at 
\texttt{\AC[a][0]} (tail) and \texttt{\AC[a][1]} (tip).
Since $4 = 2 \times 2$ and $2$ are powers of $2$, this allows the compiler to use fused add-with-shift instructions for indexing.
Formally, we may treat \texttt{\CA[c]} as a $2 \times 2$-matrix and \texttt{\AC[a]} as a $2$-vector.

Additionally, we store an $n$-vector of crossing state flags and an $m$-vector of arc state flags. 
The crossing state flags can take three states: ``left-handed'', ``right-handed'' or ``inactive''.
The arc state flags can take only the values ``active'' or ``inactive''.
At initialization, all crossing state flags are set to ``left-handed'' or ``right-handed'' according to the respective crossing's handedness, and all arc state flags are set to ``active''.
Deletion of a crossing or an arc just amounts to setting the respective flag to ``inactive''; the corresponding entries from the arrays \texttt{\CA} and \texttt{\AC} will be neglected from then on.
These flags fit into a single byte. 
Storing them along with the $2 \times 2$-matrices and the $2$-vectors for crossings/arcs would ruin their nice memory alignment. 
Moreover, these flags serve as guards against loading more information:
contemporary computer architectures always fetch contiguous blocks of bytes, the so-called \emph{cache lines} at once.
On the one hand, this allows to load and check many of these state flags before even a single arc or crossing has to be fetched at all.
On the other hand, we can rely on the integrity of the data structure to navigate the active part of the diagram:
if arc \texttt{a} is active, then the crossings \texttt{\AC[a][0]} and \texttt{\AC[a][1]} must be active, too, and we can jump to them without fetching and checking their state flags; and similarly for active crossings. Hence, it is advantageous to store the connectivity information and the flags separately. As the number of active crossings is crucial information, we maintain a counter for it that we decrement whenever we delete a crossing.

\subsubsection{Ordering}
The ordering in which we label  the crossings and arcs and store them in \texttt{\CA} and \texttt{\AC} matters a lot; a bad ordering will incur a cache miss on always every memory access. 
A good one can avoid many of these misses, at least for a few frequently needed access patterns. 
Our two most important access patterns are: 
\begin{enumerate}[label={(\roman*)}]
    \item \label{item:CanonicalAccess}
    moving from one arc to its head crossing and then straight through the crossing to the opposite arc and so on; this is an access pattern very natural to knot theorists because it means ``walking along the knot''.
    \item \label{item:FaceCycleAccess}
    moving from one arc to its head or tail crossings and then turning \emph{left}; we use this mode very frequent, e.g., to cycle around (connected components of) the boundary of a face and to navigate the dual graph during the Dijkstra search.
\end{enumerate}
The default ordering we use is the ``natural'' ordering of a knot: arcs are numbered by the ``time'' of visit; crossings are ordered by the ``time'' of first visit. For multiple-component links we first traverse the link component of the first arc found in this order; then we move on to the link component of the next unvisited arc and so on. 
This is the perfect ordering for access pattern~\ref{item:CanonicalAccess}.
We call this the \emph{canonical ordering}. 
Our experiments show that this ordering is surprisingly good also for access pattern~\ref{item:FaceCycleAccess}:
for large knot diagrams (e.g., obtained from random equilateral megagons), initializing with the canonical ordering instead of a random ordering can easily lead to a speedup of factor $5$ or more for the rerouting code.

Of course, surgery and crossing/arc deletion change the ordering. Over time, this gradually reduces the cache-friendliness of our memory layout. 
We found that it is advantageous to \emph{compress} our data structure from time to time, i.e., to reinitialize it with the canonical ordering, also removing all gaps that emerged from deletion of arcs and crossings.
This requires traversal of the whole diagram, so it is expensive; but this cost amortizes quickly. 
To decide when to recompress, at a few decisive locations in our simplification pipeline we check whether more than half of the crossings are inactive; if they are, then we recompress.
Adding this recompression strategy leads to a further speedup factor of about $1.2$ for large diagrams. And for small diagrams it seems to do no harm.

\subsubsection{The order in which over- and underpasses are processed}

We process over- and underpasses in ``rounds'', interleaved with compression steps.
At the start of each round, we remember the initial value of the pass counter.
Then we look for the first active arc.
If its head goes over/under, we start growing an over-/underpass.
If its tail goes under/over in the opposite way, then we can be sure that this arc is the start of a maximal over-/underpass.
Otherwise, we are somewhere in the middle of a maximal over-/underpass and walk backwards to its start.
Once an over-/underpass has been processed (by rerouting it or by deciding that rerouting is not salient), we immediately start a new under-/overpass at the last arc of the processed pass.
That is, we walk backwards to find the start of a maximal pass containing this arc.\footnote{By walking backwards we make it likely that the accessed memory is still ``warm'', i.e., it still resides in a fast processor cache.}

Sometimes, we finish processing a pass whose last arc does not exist (because the whole pass was a loop) or whose last arc's follow-up arc has been recently rerouted.
In that case, we instead look for the next active arc that has not been rerouted recently and continue the pass search there. 
If all arcs have been ``touched'' this way, we end the round, check to see if we should recompress it, and start a new round of rerouting. If a full round has been completed without any changes in the diagram, then we are done and stop.

\subsubsection{Winged Half-edges}
Two-dimensional diagrams are often encoded as winged half-edge data structures. 
For our use case, planar diagrams for knots and links, this would amount to storing for each darc $da$:
its head crossing, its face to the left, the next left darc $\Left(da)$, and the darc $\Right(\Reverse(da))$ to the right of $\Reverse(da)$.
(Moreover, one would select and store one of the adjacent darcs per crossing and per face.)
Albeit this data structure is very useful for many mesh processing tasks (e.g., for mesh subdivision and remeshing), we found it not so well-suited for the kind of surgery we have to do in the rerouting phase: faces are split and merged quite frequently, so it is cumbersome to keep all darcs up to date.\footnote{This is why we decided to not track face information explicitly.}
Moreover, for the Dijkstra search on $\hat{D}^*$ we need neither know what the crossings are nor what $\Right(\Reverse(da))$ is.
All we need to know is $\Left(da)$; 
carrying along any further data, loading it into cache, and maintaining it would just slow us down.

We could compute $\Left(da)$ on the fly by looking up the head crossing of $da$ in \texttt{\AC} first and then looking up the left darc in \texttt{\CA}.
However, these are two indirections, meaning two opportunities for cache misses.
We can reduce this to a single lookup by precomputing an array \texttt{\AleftA} of size $2 \, m = 4 \, n$ whose $da$-th entry is $\Left(da)$, and by indexing into that.
The lookups in 
this array
profit from a good ordering of arcs (which implies a similar ordering for darcs). 
This way, ordering of the crossings becomes irrelevant for the shortest path search.

The array \texttt{\AleftA} can be computed quickly by cycling over the active \emph{crossings} of the diagram: each crossing can tell the four incoming darcs who their next left darc is (it's the reverses of the incoming darcs). 
Moreover, \texttt{\AleftA} is easy to update when any surgery is done. 
So 
this array
needs to be computed only once for all pass moves (or when a recompression happens).

Since the shortest path search is memory bound, we indeed exhibit speedups close to a factor of~$2$ in practice (e.g., factor of about $1.8$ for diagrams obtained from megagons).

\subsubsection{Arrays vs.\ Hash Maps}

Markers and backtracking information on arcs for the shortest path search can be stored in various ways.
As we said before, we just used ordinary, contiguous, heap-allocated arrays of size $m$ for this and uses a time stamp (the current value of the overpass count) for fast reset.

Another popular way to store information assigned to arcs is  to use associative containers keyed on the arcs' labels.
There are various implementations for such data structures. 
One way is to use a heap, i.e., a binary sort tree into which the key-value pairs are inserted as nodes. 
However, node insertion and node lookup cost $O(\log(k))$ on average, where $k$ is the size of the heap. 
Moreover, depending on the implementation, deletion of the heap might cost $O(k)$ if the nodes need to be deallocated one by one. 
This is why we did not consider this data structure at all.

Another class of associative containers are hash maps; these offer insertions and lookups with cost $O(1)$.
We tried the following C++ implementations: 
\texttt{std::unordered\textunderscore{}map} 
from the \emph{C++ standard library}; \texttt{boost::unordered\textunderscore{}map}, and \texttt{boost::unordered\textunderscore{}flat\textunderscore{}map} 
of the \emph{boost} library \cite{boost}.
Of all these associative containers \texttt{boost::unordered\textunderscore{}flat\textunderscore{}map} performed best. 
However, the runtime for the pass rerouting was still three times as long as with our array implementation.
This might come as a surprise at first:
one might think that using an associative container is advantageous because the number $k$ of visited arcs during the shortest path search should be relatively small compared to the number $m$ of arcs. 
Alas, each of these containers have a cost of $O(k)$ for reset. 
And while insertions and lookups have constant cost, the cost for hashing cannot be neglected.
Moreover, the very nature of hash maps may lead to a suboptimal memory access pattern.
In contrast, our array-based implementation may profit from cache-friendly orderings of the arcs (and from repeated recompression).

\subsubsection{Flagging Proven Minimal Diagrams}

Once we have run the routine that detects and disconnects connected summands, we know that the diagram (and all its diagram components) are \emph{reduced}: they do not contain any loops or isthmii. Afterwards, the routine that splits the diagram into separate instances of \texttt{PlanarDiagram} can check if the newly created diagrams are alternating.
If yes, then the new diagram is minimal~\cite{kauffmanStateModelsJones1987,MurasugiJonesClassical,ThistlethwaiteSpanningTree}, i.e., its number of crossings coincides with the minimal crossing number of the underlying link class. In this case we set a flag in \texttt{PlanarDiagram} that excludes this diagram from any further rerouting attempts.

\section{Unknot Recognition Protocols}
\label{sec:protocols}

This appendix gives the source code for the three non-diagrammatic unknot recognition methods compared in~\cref{tab:benchmarks}, as run in our benchmarks with \texttt{tries=10}, using \SnapPy~3.3 and \Regina~7.4.

\paragraph{\SnapPy fundamental group method.} This procedure was suggested to us by Nathan Dunfield. \texttt{exterior()} builds a randomized triangulation of the knot complement. The procedure keeps the smallest of up to 11 such triangulations, calls \texttt{randomize()} on it up to 100 times until it has fewer than 10 tetrahedra, and then checks whether any relators of the simplified fundamental group remain.

\medskip
\noindent\begin{minipage}{\textwidth}
\footnotesize
\begin{verbatim}
def is_unknotted(knot_diagram, tries=10):
    E = knot_diagram.exterior()
    for i in range(tries):
        if E.num_tetrahedra() < 10:
            return len(E.fundamental_group().relators()) == 0
        F = knot_diagram.exterior()
        if F.num_tetrahedra() < E.num_tetrahedra():
            E = F

    for i in range(10*tries):
        E.randomize()
        if E.num_tetrahedra() < 10:
            return len(E.fundamental_group().relators()) == 0

    return len(E.fundamental_group().relators()) == 0
\end{verbatim}
\end{minipage}

\paragraph{\Regina triangulation and fundamental group methods.} Both follow the control flow of Dunfield's procedure line for line, with \Regina's routines in place of \SnapPy's. \Regina~7.4 has no analogue of \texttt{randomize()}, but \texttt{Triangulation3.simplify()} makes random 4-4 and book moves before trying to reduce, so repeated calls are independent random attempts.

\medskip
\noindent\begin{minipage}{\textwidth}
\footnotesize
\begin{verbatim}
def is_unknotted_regina_triangulation(link, tries=10):
    """Regina, triangulation route.  Conclusive: True and False are both proofs."""
    # fresh seed per diagram (see note below)
    regina.RandomEngine.reseedWithHardware()
    # Reidemeister I and II moves only
    link.simplifyToLocalMinimum()
    if link.size() == 0:
        return True
    # ideal triangulation, simplified on construction
    T = link.complement()
    # complement() is randomised: rebuild, keep smallest
    for i in range(tries):
        if T.size() < 10:
            return T.isSolidTorus()
        U = link.complement()
        if U.size() < T.size():
            T = U
    # still large: retry the randomised simplifier
    for i in range(10*tries):
        T.simplify()
        if T.size() < 10:
            return T.isSolidTorus()
    return T.isSolidTorus()
\end{verbatim}
\end{minipage}

\medskip
\noindent\begin{minipage}{\textwidth}
\footnotesize
\begin{verbatim}
def is_unknotted_regina_group(link, tries=10):
    """Regina, group route.  One-sided: True is a proof, False is inconclusive."""
    regina.RandomEngine.reseedWithHardware()
    link.simplifyToLocalMinimum()
    if link.size() == 0:
        return True
    T = link.complement()
    for i in range(tries):
        if T.size() < 10:
            return has_trivial_presentation(T)
        U = link.complement()
        if U.size() < T.size():
            T = U
    for i in range(10*tries):
        T.simplify()
        if T.size() < 10:
            return has_trivial_presentation(T)
    return has_trivial_presentation(T)

def has_trivial_presentation(T):
    """True iff the fundamental group of T simplifies to < a >,
    i.e. the group is Z."""
    # presentation from the dual 2-skeleton
    G = regina.GroupPresentation(T.group())
    # small cancellation + Nielsen moves, to a fixed point
    G.simplify()
    if G.countRelations() == 0:
        return True
    # one more small-cancellation pass
    G.smallCancellation()
    return G.countRelations() == 0
\end{verbatim}
\end{minipage}
\medskip

\Regina seeds its random number generator with a fixed default in every new process, and \texttt{complement()}, \texttt{simplify()} and \texttt{isSolidTorus()} all make random choices. Without the reseed in the first line of each function, a diagram's result would depend on that default seed and on whatever the process had done earlier; with it, each diagram's result is an independent sample of a randomized procedure, as for \SnapPy.

In \Regina~7.4, \texttt{Triangulation3.group()} already simplifies the presentation it returns, iterating small cancellation and Nielsen moves until neither changes it, so the explicit \texttt{simplify()} and \texttt{smallCancellation()} calls in \texttt{has\_trivial\_presentation} change nothing. They keep the procedure correct if a future version of \Regina stops simplifying inside \texttt{group()}.

\end{document}